\documentclass[a4paper,leqno,10pt]{amsart}

\usepackage{epsf,graphicx,epsfig}
\usepackage{amscd}
\usepackage{amsmath,latexsym,amssymb,amsthm}
\usepackage[nospace,noadjust]{cite}
\usepackage{textcomp}
\usepackage{setspace,cite}
\usepackage{lscape,fancyhdr,fancybox}
\usepackage{stmaryrd}
\usepackage[all,cmtip]{xy}
\usepackage{tikz}
\usepackage{cancel}
\usetikzlibrary{shapes,arrows,decorations.markings}
\usepackage{graphicx}

\usepackage{color}
\usepackage{url}
\usepackage{enumerate}
\usepackage[mathscr]{euscript}

  \theoremstyle{plain}

\swapnumbers
    \newtheorem{thm}{Theorem}[section]
    \newtheorem{prop}[thm]{Proposition}

    \newtheorem{subsec}[thm]{}
\theoremstyle{definition}
    \newtheorem{defn}[thm]{Definition}
        \newtheorem{remark}[thm]{Remark}
    \newtheorem{exam}[thm]{Example}

    \newtheorem{notation}[thm]{Notation}

\theoremstyle{remark}

\newcommand{\R}{\mathbb{R}}

\title{}
\author{}
\date{}
\usepackage{amssymb}

\usepackage{hyperref}
\hypersetup{
	colorlinks,
	citecolor=blue,
	filecolor=black,
	linkcolor=blue,
	urlcolor=black
}

\begin{document}
\title{Bimodules over relative Rota-Baxter algebras and cohomologies}

\author{Apurba Das}
\address{Department of Mathematics,
Indian Institute of Technology Kharagpur, Kharagpur-721302, West Bengal, India.}
\email{apurbadas348@gmail.com, apurbadas348@maths.iitkgp.ac.in}

\author{Satyendra Kumar Mishra}
\address{Institute for Advancing Intelligence, TCG-Centres for Research and Education
in Science and Technology, Calcutta, West Bengal, India.}
\email{satyendra.mishra@tcgcrest.org; satyamsr10@gmail.com}



\begin{abstract}
A relative Rota-Baxter algebra is a generalization of a Rota-Baxter algebra. Relative Rota-Baxter algebras are closely related to dendriform algebras. In this paper, we introduce bimodules over a relative Rota-Baxter algebra that fits with the representations of dendriform algebras. We define the cohomology of a relative Rota-Baxter algebra with coefficients in a bimodule and then study abelian extensions of relative Rota-Baxter algebras in terms of the second cohomology group. Finally, we consider homotopy relative Rota-Baxter algebras and classify skeletal homotopy relative Rota-Baxter algebras in terms of the above-defined cohomology.
\end{abstract}


\maketitle





\noindent

\thispagestyle{empty}

\tableofcontents

\vspace{0.2cm}

\section{Introduction}



Rota-Baxter operators were first appeared in a 1960's paper of Baxter \cite{baxter} as a tool to study fluctuation theory in probability. Subsequently, such operators were investigated by Rota \cite{rota}, Cartier \cite{cartier}, Atkinson \cite{atkinson} among others. In the last twenty years, Rota-Baxter operators received very much attention due to their connection with many branches of mathematics and mathematical physics. For instance, Rota-Baxter operators play a key role in the combinatorial study of trees and shuffles \cite{guo-adv}, splitting of algebras \cite{bai-splitting}, infinitesimal bialgebras \cite{aguiar-pre}, Hopf algebras \cite{zgg}, double algebras \cite{gon-kol} and renormalizations in quantum field theory \cite{conn}. 
Let $A$ be an associative algebra. A linear map $R: A \rightarrow A$ is said to be a Rota-Baxter operator on $A$ if
\begin{align*}
R(a) \cdot R(a') = R \big( R(a) \cdot a' ~+~ a \cdot R(a') \big), \text{ for } a, a' \in A.
\end{align*}
A pair $(A, R)$ consisting of an associative algebra $A$ and a Rota-Baxter operator $R$ is called a Rota-Baxter algebra. The notion of Rota-Baxter bimodules over a Rota-Baxter algebra was introduced by Guo and Lin in \cite{GuoLin}. Recently, Wang and Zhou \cite{wang-zhou} defined the cohomology of a Rota-Baxter algebra with coefficients in a Rota-Baxter bimodule. Such cohomology is closely related to abelian extensions and deformations of Rota-Baxter algebras. See \cite{guo-book} for more details about Rota-Baxter operators.

\medskip

In \cite{uchino} Uchino introduced a notion of generalized Rota-Baxter operator as an operator analogue of Poisson structures in the noncommutative setup. Later, generalized Rota-Baxter operators are also named as relative Rota-Baxter operators or $\mathcal{O}$-operators \cite{bai-o,das-rota}. They are mostly known as relative Rota-Baxter operators in the literature. Let $A$ be an associative algebra and $M$ be an $A$-bimodule. A linear map $R : M \rightarrow A$ is said to be a relative Rota-Baxter operator if
\begin{align*}
R(m) \cdot R(m') = R (R(m) \cdot_M m' ~+ ~ m \cdot_M R(m')),~\text{ for } m, m' \in M.
\end{align*}
Here $\cdot_M$ denotes the left and right $A$-actions on $M$. Recently, relative Rota-Baxter operators are cohomologically studied in \cite{das-rota}.
A relative Rota-Baxter (associative) algebra is a triple $(A, M, R)$ in which $A$ is an associative algebra, $M$ is an $A$-bimodule, and $R: M \rightarrow A$ is a relative Rota-Baxter operator. In the present paper, we denote a relative Rota-Baxter algebra as above by $M \xrightarrow{R} A$. It follows that any Rota-Baxter algebra $(A,R)$ can be considered as a relative Rota-Baxter algebra $A \xrightarrow{R} A$, where $A$ is equipped with the adjoint $A$-bimodule structure. It has been observed in \cite{uchino} that a relative Rota-Baxter algebra $M \xrightarrow{R} A$ induces a dendriform structure on $M$.

\medskip


Rota-Baxter operators and relative Rota-Baxter operators in the context of Lie algebras first appeared in the study of classical $r$-matrices \cite{kuper}. Cohomologies of relative Rota-Baxter operators on Lie algebras were initially studied in \cite{tang}. Then using higher derived brackets, the cohomology theory of relative Rota-Baxter Lie algebras were successfully established by Lazarev, Sheng and Tang in \cite{laza-rota}. Recently, cohomologies of relative Rota-Baxter Lie algebras with coefficients in an arbitrary representation were given by Jiang and Sheng in \cite{jiang-sheng}. Applications of these cohomologies were given to study infinitesimal deformations, abelian extensions and homotopies of relative Rota-Baxter Lie algebras (we refer to \cite{Guan, Guan2} for deformations and cohomologies of algebraic structures). Some aspects of the above study are already considered for relative Rota-Baxter (associative) algebras by the present authors \cite{DasSK}. More precisely, we defined cohomologies of relative Rota-Baxter algebras, which can be used to characterize infinitesimal deformations. Homotopy relative Rota-Baxter algebras were also given in the same paper. However, cohomologies of relative Rota-Baxter algebras are not yet studied from the representation-theoretic points of view.

\medskip

The main aim of this paper is to define bimodules over a relative Rota-Baxter algebra and introduce the cohomology of a relative Rota-Baxter algebra with coefficients in a bimodule. Let $M \xrightarrow{R} A$ be a relative Rota-Baxter algebra. A bimodule over it consists of a $2$-term chain complex $N \xrightarrow{S} B$ in which $B$ and $N$ are both $A$-bimodules, together with two bilinear maps $l:M\otimes B\rightarrow N$ and $r:B\otimes M\rightarrow N$ satisfying some compatibility relations (see Definition \ref{rRB-bimodule}). Note that our notion of bimodule over a relative Rota-Baxter algebra generalizes the Rota-Baxter bimodule over a Rota-Baxter algebra (see Example \ref{rb-bim}). Further, any relative Rota-Baxter algebra can be regarded as a bimodule over itself, called the adjoint bimodule. Given a bimodule over a relative Rota-Baxter algebra, we construct the corresponding dual bimodule. 

\medskip

Moreover, there is a semidirect product construction associated with a bimodule over a relative Rota-Baxter algebra (cf. Proposition \ref{semi-rrb}). Our notion of bimodule over a relative Rota-Baxter algebra nicely fits with the representations of dendriform algebras. More precisely, we show that a bimodule over a relative Rota-Baxter algebra gives rise to a representation of the induced dendriform algebra (cf. Proposition \ref{induced-dend-repn}). 

\medskip

Next, we introduce the cohomology of a relative Rota-Baxter algebra with coefficients in a bimodule. The cohomology of a relative Rota-Baxter algebra introduced in \cite{DasSK} turns out to be the cohomology of the relative Rota-Baxter algebra with coefficients in the adjoint bimodule. Then we consider abelian extensions of a relative Rota-Baxter algebra by a bimodule.  We show that isomorphism classes of such abelian extensions are classified by the second cohomology group (cf. Theorem \ref{abelian-ext-thm}).

\medskip


The notion of homotopy relative Rota-Baxter operators in the context of $A_\infty$-algebras (strongly homotopy associative algebras) was introduced in \cite{DasSK}. Like the classical case, a homotopy relative Rota-Baxter algebra consists of an $A_\infty$-algebra $\mathcal{A}$, an $\mathcal{A}$-bimodule $\mathcal{M}$ and a homotopy relative Rota-Baxter operator. In this paper, we consider homotopy relative Rota-Baxter algebras whose underlying graded vector spaces $\mathcal{A}, \mathcal{M}$ are concentrated in degrees $0$ and $1$. In particular, we classify skeletal homotopy relative Rota-Baxter algebras by third cocycles of relative Rota-Baxter algebras introduced in the present paper (cf. Theorem \ref{skeletal-thm}).

\medskip

The paper is organized as follows. In Section \ref{sec-2}, we recall the Hochschild cohomology theory of associative algebras and recall relative Rota-Baxter algebras. In Section \ref{sec-3}, we define bimodules over relative Rota-Baxter algebras and give various examples and constructions. The cohomology of a relative Rota-Baxter algebra with coefficients in a bimodule is introduced in Section \ref{sec-4}. In Section \ref{sec-5}, we study abelian extensions of relative Rota-Baxter algebras in terms of the second cohomology group. Finally, classification of skeletal homotopy relative Rota-Baxter algebras is given in Section \ref{sec-6}.

\medskip

All vector spaces, (multi)linear maps, and tensor products are over a field ${\bf k}$ of characteristic $0$.



\section{Preliminaries}\label{sec-2}
This section first recalls the Hochschild cohomology of associative algebra with coefficients in a bimodule. Then we recall relative Rota-Baxter algebras, which are the main objective in this paper, and their relation with dendriform algebras. Our main references are \cite{loday-cyclic,das-rota,uchino}.

\medskip

Let $(A, \mu)$ be an associative algebra. For any two elements $a, a' \in A$, we write the multiplication $\mu (a, a')$ by $a \cdot a'$. An associative algebra $(A, \mu)$ is often denoted by $A$.

Let $A$ be an associative algebra. An $A$-bimodule is a vector space $M$ together with two bilinear maps (called left and right $A$-actions) $l_M : A \otimes M \rightarrow M,~(a,m) \mapsto a \cdot_M m$ and $r_M : M \otimes A \rightarrow M,~ (m,a) \mapsto m \cdot_M a$ satisfying
\begin{align}\label{bimod}
(a \cdot a') \cdot_M m = a \cdot_M (a' \cdot_M m),  \qquad (a \cdot_M m) \cdot_M a' = a \cdot_M (m \cdot_M a') \quad \text{ and } \quad (m \cdot_M a) \cdot_M a' = m \cdot_M (a \cdot a'),
\end{align}
for $a, a' \in A$ and $m \in M$. An $A$-bimodule as above is simply denoted by $M$ when the left and right $A$-actions are clear from the context. Any associative algebra $A$ is obviously an $A$-bimodule with left and right $A$-actions are given by the algebra multiplication map. We call this the adjoint $A$-bimodule.

Given an $A$-bimodule $M$, the dual space $M^*$ also carries an $A$-bimodule structure with left, and right $A$-actions are given by
\begin{align*}
(a \cdot_{M^*} f)(m) = f (m \cdot_M a) ~~~~ \text{ and } ~~~~ (f \cdot_{M^*} a)(m) = f (a \cdot_M m), ~ \text{ for } a \in A,~ f \in M^*, ~m \in M.
\end{align*}
The above bimodule structure is called the dual $A$-bimodule structure on $M^*$. In particular, the dual of the adjoint $A$-bimodule is called the coadjoint $A$-bimodule.


\medskip

Let $A$ be an associative algebra and $M$ be an $A$-bimodule. Then the direct sum $A \oplus M$ carries an associative algebra structure with the multiplication given by
\begin{align*}
(a,m) \cdot_\ltimes (a', m') = (a \cdot a',~ a \cdot_M m' + m \cdot_M a'),~ \text{ for } (a,m), (a', m') \in A \oplus M.
\end{align*}
The pair $(A \oplus M,\cdot_{\ltimes})$ is called the semidirect product algebra. The semidirect product algebra will play a crucial role in our study. 
\medskip

We will now recall the Hochschild cohomology of an associative algebra with coefficients in a bimodule.
Let $A$ be an associative algebra and $M$ be an $A$-bimodule. For each $k \geq 0$, define an abelian group $C^k_\mathsf{H} (A, M) := \mathrm{Hom} (A^{\otimes k}, M)$. Moreover, there is a differential $\delta_{A,M} : C^k_\mathsf{H} (A, M) \rightarrow C^{k+1}_\mathsf{H} (A, M)$, for $k \geq 0$, given by
\begin{align*}
(\delta_{A,M} f )(a_1, \ldots, a_{k+1}) =~& a_1 \cdot_M f (a_2, \ldots, a_{k+1}) + \sum_{i=1}^k (-1)^i ~ f (a_1, \ldots, a_i \cdot a_{i+1}, \ldots, a_{k+1}) \\
&+ (-1)^{k+1}~ f (a_1, \ldots, a_{k}) \cdot_M a_{k+1} ,
\end{align*}
for $f \in C^k_\mathsf{H}(A, M)$ and $a_1, \ldots, a_{k+1} \in A$. The cohomology groups of the cochain complex $\{ C^\bullet_\mathsf{H}(A, M), \delta_{A,M} \}$ are called the {\bf Hochschild cohomology} groups of $A$ with coefficients in $M$.

\medskip



\begin{defn}
(i) Let $A$ be an associative algebra and $M$ be an $A$-bimodule. A linear map $R : M \rightarrow A$ is said to be a {\bf relative Rota-Baxter operator} (on $M$ over the algebra $A$) if $R$ satisfies
\begin{align}\label{bimod-rr}
R(m ) \cdot R (m') = R \big( R(m) \cdot_M m' + m \cdot_M R(m')  \big), ~ \text{ for } m, m' \in M.
\end{align}

\medskip

(ii) A {\bf relative Rota-Baxter algebra} is a triple $(A, M, R)$ consisting of an associative algebra $A$, an $A$-bimodule $M$ and a relative Rota-Baxter operator $R : M \rightarrow A$.
\end{defn}

\begin{notation}
For some conventional reason, we denote a relative Rota-Baxter algebra by $M \xrightarrow{R} A$ instead of the triple $(A, M, R)$. Note that both of these notations capture the same pieces of information.
\end{notation}

\begin{exam}
Let $\mathcal{A} = ({A}_1 \xrightarrow{d} {A}_0)$ be a $2$-term chain complex. Define $\mathrm{End}(\mathcal{A})$ to be the space of all chain maps on $\mathcal{A}$, i.e.,
\begin{align*}
\mathrm{End}(\mathcal{A}) = \{ (f_0, f_1) ~|~ f_0 \in \mathrm{End}({A}_0),~ f_1 \in \mathrm{End}({A}_1) \text{ and } f_0 \circ d = d \circ f_1 \}.
\end{align*}
Note that $\mathrm{End}(\mathcal{A})$ carries an associative algebra structure with multiplication given by
\begin{align*}
(f_0, g_0) \cdot (g_0, g_1) = (f_0 \circ g_0, f_1 \circ g_1), ~ \text{ for } (f_0, g_0), (g_0, g_1) \in \mathrm{End}(\mathcal{A}).
\end{align*}
Moreover, there is a $\mathrm{End}(\mathcal{A})$-bimodule structure on the space $M = \mathrm{Hom} ({A}_0, \mathrm{ker }~ d)$ given by
\begin{align*}
(f_0, f_1) \cdot_M \phi = f_1 \circ \phi ~~~ \text{ and } ~~~ \phi \cdot_M (f_0, f_1) = \phi \circ f_0, \text{ for } (f_0, f_1) \in  \mathrm{End}(\mathcal{A}), \phi \in M = \mathrm{Hom} ({A}_0, \mathrm{ker }~ d).
\end{align*}
Finally, it is easy to verify that the map $\mathcal{R}: M \rightarrow \mathrm{End}(\mathcal{A})$, $\mathcal{R}(\phi) = (0, \phi \circ d)$ is a relative Rota-Baxter operator. In other words, $ M \xrightarrow{\mathcal{R}} \mathrm{End}(\mathcal{A})$  is a relative Rota-Baxter algebra.
\end{exam}

Some other examples of relative Rota-Baxter algebras can be found in \cite{uchino}.

\begin{defn}\label{rba-mor}
Let $M \xrightarrow{R} A$ and $N \xrightarrow{S} B$ be two relative Rota-Baxter algebras. A {\bf morphism} of relative Rota-Baxter algebras from $M \xrightarrow{R} A$ to $N \xrightarrow{S} B$ is a pair $(\phi, \psi)$ consists of an algebra morphism $\phi : A \rightarrow B$ and a linear map $\psi : M \rightarrow N$ satisfying
\begin{align*}
\psi (a \cdot_M m) = \phi (a) \cdot_N \psi (m), \qquad \psi (m \cdot_M a) = \psi (m) \cdot_N \phi (a) \quad \text{ and } \quad \phi \circ R = S \circ \psi,~ \text{ for } a \in A, m \in M.
\end{align*}
It is called an isomorphism if $\phi$ and $\psi$ are linear isomorphisms.
\end{defn}

Next, we recall dendriform algebras that are intimately connected with relative Rota-Baxter algebras \cite{loday,aguiar-pre,uchino}.

\begin{defn}
A {\bf dendriform algebra} is a vector space $D$ together with two bilinear operations $\prec, \succ : D \otimes D \rightarrow D$ satisfying
\begin{align}
(x \prec y) \prec z =~& x \prec (y \prec z + y \succ z), \label{dend-defn-1}\\
(x \succ y) \prec z =~& x \succ (y \prec z),\\
 (x \prec y + x \succ y) \succ z  =~& x \succ (y \succ z), ~ \text{ for } x, y, z \in D. \label{dend-defn-3}
\end{align}
\end{defn}

Let $(D, \prec, \succ)$ be a dendriform algebra. It follows from (\ref{dend-defn-1})-(\ref{dend-defn-3}) that the new operation $x \ast y =  x \prec y + x \succ y$, for $x, y \in D$, makes $D$ into an associative algebra. We call this the `total associative algebra', denoted by $D_{\mathrm{Tot}}$.



\begin{defn}
Let $(D, \prec, \succ)$ be a dendriform algebra. A {\bf representation} of it consists of a vector space $E$ together with four bilinear maps (called action maps)
\begin{align*}
\prec : D \otimes E \rightarrow E, \qquad \succ : D \otimes E \rightarrow E, \qquad \prec : E \otimes D \rightarrow E \quad \text{ and } \quad \succ : E \otimes D \rightarrow E
\end{align*}
satisfying the $9$ identities where each tuple of $3$ identities correspond to (\ref{dend-defn-1})-(\ref{dend-defn-3}) with exactly one of $x,y,z$ comes from $E$.
\end{defn}

\begin{remark}
(i) Any dendriform algebra $D$ is a representation of itself, where action maps are dendriform operations on $D$. It is called the adjoint representation.

(ii) Let $D$ be a dendriform algebra and $E$ be a representation of it. Then $E$ can be given a $D_\mathrm{Tot}$-bimodule structure with left and right $D_\mathrm{Tot}$-actions
\begin{align*}
 \mathsf{x} \cdot_E e = \mathsf{x} \prec e + \mathsf{x} \succ e ~~~~~ \text{ and } ~~~~~ e \cdot_E \mathsf{x} = e \prec \mathsf{x} + e \succ \mathsf{x}, ~ \text{ for } \mathsf{x} \in D_\mathrm{Tot},~ e \in E.
\end{align*}
\end{remark}

The following result can be found in \cite{uchino,das-rota}.

\begin{prop}\label{prop-rrb-dend}
(i) Let $M \xrightarrow{R} A$ be a relative Rota-Baxter algebra. Then $M$ carries a dendriform algebra structure given by
\begin{align*}
m \prec_R m' = m \cdot_M R (m') \quad ~ \text{ and } ~ \quad m \succ_R m' = R (m) \cdot_M m' , ~ \text{ for } m, m' \in M.
\end{align*}
Hence $(M, \ast_R)$ is an associative algebra, where $m \ast_R m' = R(m) \cdot_M m' + m \cdot_M R(m')$, for $m, m' \in M$. (We denote this total associative algebra by $M_\mathrm{Tot}$). With this associative structure, the map $R : M \rightarrow A$ is a morphism of associative algebras.

(ii) If $(\phi, \psi)$ is a morphism of relative Rota-Baxter algebras from $M \xrightarrow{R} A$ to $N \xrightarrow{S} B$, then the map $\psi : M \rightarrow N$ is a morphism of dendriform algebras.
\end{prop}

\section{Bimodules over relative Rota-Baxter algebras}\label{sec-3}
In this section, we introduce bimodules over relative Rota-Baxter algebras and provide various examples. We  construct the corresponding semidirect product in the context of relative Rota-Baxter algebras. Finally, we show that a bimodule over a relative Rota-Baxter algebra gives rise to a representation of the induced dendriform algebra and find out a relationship between the cohomology of a relative Rota-Baxter operator and the cohomology of the induced dendriform algebra.

\begin{defn}\label{rRB-bimodule}
Let $M \xrightarrow{R} A$ be a relative Rota-Baxter algebra. A {\bf bimodule} over it consists of a tuple $(N \xrightarrow{S} B, l, r)$ in which $N \xrightarrow{S} B$ is a $2$-term chain complex with both $B$ and $N$ are $A$-bimodules, and there are two bilinear maps  $l : M \otimes B \rightarrow N$ and $ r : B \otimes M \rightarrow N$ satisfying
\begin{align}
l (a \cdot_M m, b) = a \cdot_N l(m,b), \qquad l (m \cdot_M a, b) = l (m, a \cdot_B b), \qquad l (m, b \cdot_B a) = l (m, b) \cdot_N a, \label{bimod-l}\\
r ( a \cdot_B b, m ) =  a \cdot_N  r (b, m), \qquad r (b \cdot_B a, m) = r (b, a \cdot_M m), \qquad r (b, m \cdot_M a ) = r (b, m) \cdot_N a,  \label{bimod-r}
\end{align}
and
\begin{align}
R(m) \cdot_B S(n) =~& S \big( R(m) \cdot_N n + l (m, S(n))  \big), \label{bimod-rs}\\
S(n) \cdot_B R(m) =~& S \big(  r (S(n), m) + n \cdot_N R(m) \big), \label{bimod-sr}
\end{align}
for $a \in A,~ b \in B,~ m \in M,~ n \in N$.
\end{defn}

A bimodule as above is often denoted by the complex $N \xrightarrow{S} B$ when the bilinear maps $l,r$ are clear from the context.

\begin{exam}\label{rb-bim}({\bf Rota-Baxter bimodule})
Let $(A, R)$ be a Rota-Baxter algebra. Then it can be considered as a relative Rota-Baxter algebra $A \xrightarrow{R} A$, where $A$ is equipped with the adjoint $A$-bimodule structure. A Rota-Baxter bimodule \cite{GuoLin} over the Rota-Baxter algebra $(A, R)$ consists of a pair $(M, R_M)$ in which $M$ is an $A$-bimodule and $R_M : M \rightarrow M$ is a linear map satisfying
\begin{align*}
R(a) \cdot_M R_M (m) = R_M \big( R(a) \cdot_M m ~+~ a \cdot_M R_M (m)  \big),\\R_M(m) \cdot_M R (a) = R_M \big( R_M(m) \cdot_M a ~+~ m \cdot_M R (a)  \big),
\end{align*}
for $a \in A,~m \in M$. Then it follows that $M \xrightarrow{R_M} M$ is a bimodule over the relative Rota-Baxter algebra $A \xrightarrow{R} A$, where the bilinear maps $l$ and $r$ are respectively left and right $A$-actions on $M$. 

Here we consider an example of a Rota-Baxter bimodule. Let $A$ be an associative algebra. For an element ${\bf r} = r_{(1)} \otimes r_{(2)} \in A \otimes A$, we define the following three elements
\begin{align*}
{\bf r}_{13} {\bf r}_{12} = r_{(1)} \cdot \widetilde{r}_{(1)} \otimes \widetilde{r}_{(2)} \otimes r_{(2)}, \quad
{\bf r}_{12} {\bf r}_{23} = r_{(1)} \otimes r_{(2)} \cdot \widetilde{r}_{(1)} \otimes \widetilde{r}_{(2)}  ~~~ \text{~ and ~ } ~~~
{\bf r}_{23} {\bf r}_{13} = r_{(1)} \otimes \widetilde{r}_{(1)} \otimes \widetilde{r}_{(2)} \cdot r_{(2)} 
\end{align*}
of the tensor product $A \otimes A \otimes A$. Here ${\bf r} = \widetilde{r}_{(1)} \otimes \widetilde{r}_{(2)}$ is an another copy of ${\bf r}$. We consider the following equation, called the associative Yang-Baxter equation (in short AYBE)
\begin{align*}
{\bf r}_{13} {\bf r}_{12} - {\bf r}_{12} {\bf r}_{23} + {\bf r}_{23} {\bf r}_{13} = 0,
\end{align*}
equivalently,
\begin{align}\label{aybe-exp}
r_{(1)} \cdot \widetilde{r}_{(1)} \otimes \widetilde{r}_{(2)} \otimes r_{(2)} ~-~  r_{(1)} \otimes r_{(2)} \cdot \widetilde{r}_{(1)} \otimes \widetilde{r}_{(2)} ~+~ r_{(1)} \otimes \widetilde{r}_{(1)} \otimes \widetilde{r}_{(2)} \cdot r_{(2)} = 0.
\end{align}
A solution of the associative Yang-Baxter equation is called an associative $r$-matrix.

Let $A$ be an associative algebra and ${\bf r} = r_{(1)} \otimes r_{(2)} \in A \otimes A$ be an associative $r$-matrix. We define a map $R: A \rightarrow A$ by $R(a) = r_{(1)} \cdot a \cdot r_{(2)}$, for $a \in A$. Then it has been shown in \cite{aguiar-pre} that $R$ is a Rota-Baxter operator on $A$. In other words, $(A, R)$ is a Rota-Baxter algebra. Further, if $M$ is an $A$-bimodule, we define a map $R_M : M \rightarrow M$ by $R_M (m) = r_{(1)} \cdot_M m  \cdot_M r_{(2)}$, for $m \in M$. Then it can be checked that $(M, R_M)$ is a Rota-Baxter bimodule over the Rota-Baxter algebra $(A, R)$. To see this, we observe that
\begin{align*}
R(a) \cdot_M R_M (m) =~& ( r_{(1)} \cdot a \cdot r_{(2)} ) \cdot_M (r_{(1)} \cdot_M m  \cdot_M r_{(2)}) \\
\stackrel{(\ref{aybe-exp})}{=}~& r_{(1)} \cdot \widetilde{r}_{(1)} \cdot a \cdot \widetilde{r}_{(2)} \cdot_M m \cdot_M r_{(2)} ~+~ r_{(1)} \cdot a \cdot \widetilde{r}_{(1)} \cdot_M m \cdot_M \cdot \widetilde{r}_{(2)} \cdot r_{(2)}\\
=~& R_M \big( R(a) \cdot_M m + a \cdot_M R_M (m)  \big)
\end{align*}
and
\begin{align*}
R_M (m) \cdot_M R(a) =~& (r_{(1)} \cdot_M m  \cdot_M r_{(2)}) \cdot_M  ( r_{(1)} \cdot a \cdot r_{(2)} ) \\
\stackrel{(\ref{aybe-exp})}{=}~& r_{(1)} \cdot \widetilde{r}_{(1)} \cdot_M m \cdot_M \widetilde{r}_{(2)} \cdot a \cdot r_{(2)} ~+~  r_{(1)} \cdot_M m \cdot_M \widetilde{r}_{(1)} \cdot a \cdot \widetilde{r}_{(2)} \cdot r_{(2)} \\
=~& R_M \big( R_M (m) \cdot_M a + m \cdot_M R (a)  \big).
\end{align*}
\end{exam}

\begin{exam}\label{adj-bim} ({\bf Adjoint bimodule}) Any relative Rota-Baxter algebra $M \xrightarrow{R} A$ is a bimodule over itself, where $A$ is equipped with the adjoint $A$-bimodule, $M$ is equipped with the given $A$-bimodule, and the bilinear maps $l$ and $r$ are respectively right and left $A$-actions on $M$.
\end{exam}

\begin{prop} $(${\bf Dual bimodule}$)$ Let $M \xrightarrow{R} A$ be a relative Rota-Baxter algebra and $N \xrightarrow{S} B$ be a bimodule over it. Then $B^* \xrightarrow{-S^*} N^*$ is also a bimodule, where the $A$-bimodule structures on $N^*$ and $B^*$ are given by the dual $A$-bimodules, and the bilinear maps $l^* : M \otimes N^* \rightarrow B^*$ and $r^* : N^* \otimes M \rightarrow B^*$ are given by
\begin{align*}
l^* (m, f_N) (b) = f_N (r (b,m)) ~~~~ \text{ and } ~~~~ r^*(f_N, m) (b) = f_N (l(m,b)), ~\text{ for } m \in M, ~f_N \in N^* \text{ and } b \in B. 
\end{align*}
\end{prop}

\begin{proof}
We only need to check the identities of (\ref{bimod-l}), (\ref{bimod-r}), (\ref{bimod-rs}) and (\ref{bimod-sr}) for the above dual structures. For any $a \in A, m \in M, f_N \in N^*$ and $b \in B$, we have
\begin{align*}
l^* (a \cdot_M m, f_N) (b) = f_N (r (b, a \cdot_M m)) \stackrel{(\ref{bimod-r})}{=} f_N (r (b \cdot_B a, m)) = l^* (m, f_N) (b \cdot_B a)
= (a \cdot_{B^*} l^*(m, f_N)) (b),
\end{align*}
\begin{align*}
l^* (m \cdot_M a, f_N) (b) = f_N (r (b, m \cdot_M a)) \stackrel{(\ref{bimod-r})}{=} f_N (r (b ,m) \cdot_N a) = (a \cdot_{N^*} f_N) (r (b,m)) 
= l^* (m, a \cdot_{N^*} f_N)(b),
\end{align*}
\begin{align*}
l^* (m, f_N \cdot_{N^*} a)(b) = f_N (a \cdot_N r (b,m)) \stackrel{(\ref{bimod-r})}{=} f_N (r (a \cdot_B b, m)) 
= l^* (m, f_N) (a \cdot_B b)  = (l^* (m, f_N) \cdot_{B^*} a)(b).
\end{align*}
This shows that the identities in (\ref{bimod-l}) are hold. Similarly, one can verify the identities in (\ref{bimod-r}). Finally, for any $m \in M,~ f_B \in B^*$ and $n \in N$, we observe that
\begin{align*}
\big(  R(m) \cdot_{N^*} (-S^*)(f_B) \big) (n) = (- S^*)(f_B) (n \cdot_{N} R(m)) =~& - f_B \big(  S (n \cdot_N R(m)) \big) \\
\stackrel{(\ref{bimod-sr})}{=}~& - f_B \big(    S(n) \cdot_B R(m) - S \circ r (S(n), m) \big)\\
=~& - (R(m) \cdot_{B^*} f_B) (S(n)) ~+~ l^* (m, S^* (f_B)) (S(n)) \\
=~& (-S^*) \big(  R(m) \cdot_{B^*} f_B ~+~ l^* (m, -S^*(f_B))  \big)(n)
\end{align*}
and
\begin{align*}
\big(    (-S^*)(f_B) \cdot_{N^*} R(m) \big) (n) = (-S^*) (f_B) (R(m) \cdot_N n) =~& -f_B \big( S (R(m) \cdot_N n) \big) \\
\stackrel{(\ref{bimod-rs})}{=}~& f_B \big( S \circ l (m, S(n)) - R(m) \cdot_B S(n)    \big) \\
=~& S^*(f_B) (l (m, S(n))) - f_B (R(m) \cdot_B S(n)) \\
=~& r^* (S^*(f_B), m)(S(n)) - (f_B \cdot_{B^*} R(m)) (S(n)) \\
=~& (-S^*) \big( r^*(-S^* (f_B), m) + f_B \cdot_{B^*} R(m)   \big)(n).
\end{align*}
This shows that (\ref{bimod-rs}) and (\ref{bimod-sr}) also holds for the above dual structures. Hence $(B^* \xrightarrow{-S^*} N^*, l^*, r^*)$ is a bimodule over the relative Rota-Baxter algebra $M \xrightarrow{R} A.$
\end{proof}

\begin{exam} \label{coadj-bim} ({\bf Coadjoint bimodule}) This example is dual to the adjoint bimodule given in Example \ref{adj-bim}. Let $M \xrightarrow{R} A$ be a relative Rota-Baxter algebra. Then $(A^* \xrightarrow{-R^*} M^*, l^*, r^*)$ is a bimodule, where the bilinear maps $l^* : M \otimes M^* \rightarrow A^*$ and $r^* : M^* \otimes M \rightarrow\ A^*$ are given by
\begin{align*}
l^* (m, f_M) (a) = f_M (a \cdot_M m) ~~~~ \text{~ and ~} ~~~~ r^* (f_M, m)(a) = f_M (m \cdot_M a), ~ \text{ for } m \in M,~ f_M \in M^*,~ a \in A.
\end{align*}
\end{exam}

\begin{exam} ({\bf Bimodule induced by a morphism}) Let $M \xrightarrow{R} A$ and $N \xrightarrow{S} B$ be two relative Rota-Baxter algebras and $(\phi, \psi)$ be a morphism between them (see Definition \ref{rba-mor}). Then $N \xrightarrow{S} B$ can be given a bimodule structure over the relative Rota-Baxter algebra $M \xrightarrow{R} A$, where the $A$-bimodule structures on $B$ and $N$ are respectively given by
\begin{align*}
a \cdot_B b = \phi (a) \cdot b, \quad b \cdot_B a = b \cdot \phi (a)  ~~ \text{ and } ~~ a \cdot_N n = \phi (a) \cdot_N^B n, \quad n \cdot_N a = n \cdot_N^B \phi (a),
\end{align*}
for $a \in A, b \in B, n \in N$. Here $\cdot_N^B$ denote the left and right $B$-actions on $N$. Finally, the bilinear maps $l : M \otimes B \rightarrow N$ and $r : B \otimes M \rightarrow N$ are given by
\begin{align*}
l (m,b) = \psi (m) \cdot_N^B b  ~~~ \text{ and } ~~~ r (b,m) =  b \cdot_N^B \psi (m), ~~~  \text{ for } m \in M, b \in B.
\end{align*}
\end{exam}

\begin{exam}
Let $A$ be an associative algebra and $M$ be an $A$-bimodule. A linear map $d : A \rightarrow M$ is said to be a derivation on $A$ with values in $M$ if
\begin{align*}
d (a \cdot b) = a \cdot_M d (b) + d(a) \cdot_M b, ~\text{ for } a, b \in A.
\end{align*}
In this case, we call the triple $(A, M, d)$ a relative differential algebra. A bimodule over a relative differential algebra $(A, M, d)$ consists of a quintuple $(B, N, \delta, l, r)$ in which $B$ and $N$ are both $A$-bimodules, $l : M \otimes B \rightarrow N$ and $r : B \otimes M \rightarrow N$ are bilinear maps satisfying the identities in (\ref{bimod-l}),(\ref{bimod-r}), and $\delta : B \rightarrow N$ is a linear map such that followings are hold:
\begin{align*}
\delta (a \cdot_B b) = a \cdot_N \delta(b) + l (d(a), b) \quad \text{ and } \quad \delta ( b \cdot_B a) = r (b, d(a)) +  \delta (b) \cdot_N a, ~ \text{ for } a \in A, b \in B. 
\end{align*}
Let $(A, M, d)$ be a relative differential algebra and $(B, N, \delta, l, r)$ be a bimodule over it. Suppose $\mathrm{dim} ~ M = \mathrm{dim} ~ A$ and $\mathrm{dim} ~ N = \mathrm{dim} ~ B$, and the maps $d, \delta$ are invertible. Then it is easy to see that $M \xrightarrow{d^{-1}} A$ is a relative Rota-Baxter algebra and $(N \xrightarrow{\delta^{-1}} B, l, r)$ is a bimodule over it.
\end{exam}

\medskip

In \cite{uchino} Uchino observed that a relative Rota-Baxter operator lifts to a Rota-Baxter operator on the semidirect product algebra. More precisely, let $A$ be an associative algebra and $M$ be an $A$-bimodule. Consider the semidirect product algebra structure on $A \oplus M$ with the product
\begin{align*}
(a,m) \cdot_\ltimes (a', m') = (a \cdot a', ~a \cdot_M m' + m \cdot_M a'), \text{ for } (a,m), (a', m') \in A \oplus M.
\end{align*}
A linear map $R : M \rightarrow A$ is a relative Rota-Baxter operator if and only if the map $\widehat{R} : A \oplus M \rightarrow A \oplus M$ defined by $\widehat{R}(a, m) = (R(m), 0)$ is a Rota-Baxter operator on the semidirect product algebra $A \oplus M$. In the following, we show that a bimodule over a relative Rota-Baxter algebra can also be lifted.

\begin{prop}\label{prop-lift}
Let $M \xrightarrow{R} A$ be a  relative Rota-Baxter algebra and $N \xrightarrow{S} B$ be a $2$-term chain complex in which $B$ and $N$ are both $A$-bimodules. Suppose there are maps $l : M \otimes B \rightarrow N$ and $r: B \otimes M \rightarrow N$ satisfying the conditions of (\ref{bimod-l}), (\ref{bimod-r}). Then $B \oplus N$ carries an $(A \oplus M)$-bimodule structure with left and right actions
\begin{align*}
(a,m) \cdot_{B \oplus N} (b,n) =~& (a \cdot_B b, ~ a \cdot_N n + l (m, b)),\\
(b,n) \cdot_{B \oplus N} (a,m) =~& ( b \cdot_B a,~ r(b,m) + n \cdot_N a).
\end{align*}
Moreover, 
$(N \xrightarrow{S} B, l, r)$ is a bimodule over the relative Rota-Baxter algebra $M \xrightarrow{R} A$ if and only if the pair $(B \oplus N, \widehat{ S})$ is a Rota-Baxter bimodule over the Rota-Baxter algebra $(A \oplus M, \widehat{R})$. Here $\widehat{S} : B \oplus N \rightarrow B \oplus N$ is the map $\widehat{S}(b,n) = (S(n), 0)$, for $(b,n)\in B \oplus N$.
\end{prop}

\begin{proof}
The first part is a straightforward calculation. To prove the second part, we observe that
\begin{align*}
\widehat{R}(a,m) \cdot_{B \oplus N} \widehat{S}(b,n) = (R(m), 0) \cdot_{B \oplus N} (S(n), 0) = (R(m) \cdot_B S(n), 0)
\end{align*}
and
\begin{align*}
\widehat{S} \big(  \widehat{R}(a,m) \cdot_{B \oplus N} (b,n) ~+~ (a,m) \cdot_{B \oplus N} \widehat{S}(b,n) \big) 
=~& \widehat{S} \big(  (R(m), 0) \cdot_{B \oplus N} (b, n) ~+~ (a,m) \cdot_{B \oplus N} (S(n), 0) \big) \\
=~& \widehat{S} \big(   R(m) \cdot_B b + a \cdot_B S(n), ~ R(m) \cdot_N n + l (m, S(n)) \big) \\
=~& \big( S ( R(m) \cdot_N n + l (m, S(n)) ), 0   \big).
\end{align*}
This shows that $\widehat{R}(a,m) \cdot_{B \oplus N} \widehat{S}(b,n) = \widehat{S} \big(  \widehat{R}(a,m) \cdot_{B \oplus N} (b,n) ~+~ (a,m) \cdot_{B \oplus N} \widehat{S}(b,n) \big) $ holds if and only if (\ref{bimod-rs}) holds. Similarly, one can verify that $$\widehat{S}(b,n) \cdot_{B \oplus N} \widehat{R}(a,m) = \widehat{S} \big( \widehat{S}(b,n) \cdot_{B \oplus N} (a,m) + (b,n) \cdot_{B \oplus N} \widehat{R}(a,m)  \big)$$ holds if and only if (\ref{bimod-sr}) holds. This implies that $(B \oplus N, \widehat{S})$ is a Rota-Baxter bimodule over the Rota-Baxter algebra $(A \oplus M, \widehat{R})$ if and only if $(N \xrightarrow{S} B, l, r)$ is a bimodule over the relative Rota-Baxter algebra $M \xrightarrow{R} A$.
\end{proof}

\medskip

Let $M \xrightarrow{R} A$ be a relative Rota-Baxter algebra and $(N \xrightarrow{S} B, l, r)$ be a bimodule over it. Consider the direct sum $A \oplus B$ with the semidirect product algebra structure given by
\begin{align*}
(a,b) \cdot_\ltimes (a', b') = (a \cdot a', ~ a \cdot_B b' + b \cdot_B a'), ~ \text{ for } (a, b), (a', b') \in A \oplus B.
\end{align*}
We claim that the space $M \oplus N$ carries a bimodule structure over the associative algebra $A \oplus B$ with left and right actions
\begin{align*}
(a,b) \vartriangleright (m,n) = (a \cdot_M m, ~ a \cdot_N n + r (b,m)) ~~~~  \text{ and } ~~~~ (m,n) \vartriangleleft (a,b) = (m \cdot_M a, ~l (m, b) + n \cdot_N a),
\end{align*}
for $(a,b) \in A \oplus B$ and $(m,n) \in M \oplus N$. To see this, we observe that
\begin{align*}
&\big(  (a,b) \cdot_\ltimes (a', b') \big) \vartriangleright (m,n) - (a,b) \vartriangleright \big( (a', b') \vartriangleright (m,n)   \big) \\
&= (a \cdot a', ~a \cdot_B b' + b \cdot_B a') \vartriangleright (m,n)  ~-~ (a,b) \vartriangleright (a' \cdot_M m, ~ a' \cdot_N n + r (b', m))\\
&= \big(  (a \cdot a') \cdot_M m - a \cdot_M (a' \cdot_M m), ~ (a \cdot a') \cdot_N n + r (a \cdot_B b', m) + r (b \cdot_B a', m) \\
 & \quad - a \cdot_N (a' \cdot_N n) - a \cdot_N r (b', m) - r (b, a' \cdot_M m)   \big) \stackrel{(\ref{bimod}), (\ref{bimod-r})}{=} 0.
\end{align*}
This proves that $\big(  (a,b) \cdot_\ltimes (a', b') \big) \vartriangleright (m,n) = (a,b) \vartriangleright \big( (a', b') \vartriangleright (m,n)   \big)$. Similarly, we have
\begin{align*}
& \big( (a,b) \vartriangleright (m,n)    \big)\vartriangleleft (a', b') - (a, b) \vartriangleright \big(   (m,n) \vartriangleleft (a', b') \big)\\
& = (a \cdot_M m, ~ a \cdot_N n + r (b, m)) \vartriangleleft (a', b') ~-~ (a, b) \vartriangleright (m \cdot_M a', ~ l (m, b') + n \cdot_N a' )\\
& = \big( (a \cdot_M m) \cdot_M a' - a \cdot_M (m \cdot_M a'), ~ l (a \cdot_M m, b') + (a \cdot_N n) \cdot_N a' + r (b, m) \cdot_N a' \\
 & \quad - a \cdot_N l (m, b')  - a \cdot_N (n \cdot_N a') - r (b, m \cdot_M a')   \big) \stackrel{(\ref{bimod}), (\ref{bimod-l}), (\ref{bimod-r})}{=} 0
\end{align*}
and
\begin{align*}
& \big( (m,n) \vartriangleleft (a,b)  \big) \vartriangleleft (a', b') -  (m,n) \vartriangleleft \big( (a, b) \cdot_\ltimes (a', b') \big)\\
& = (m \cdot_M a, ~ l (m, b) + n \cdot_N a) \vartriangleleft (a', b') ~-~ (m,n) \vartriangleleft (a \cdot a',~ a \cdot_B b' + b \cdot_B a')\\
& = \big(  (m \cdot_M a) \cdot_M a' - m \cdot_M (a \cdot a'), ~ l (m \cdot_M a, b') + l (m, b) \cdot_N a' + (n \cdot_N a) \cdot_N a' \\
& \quad  - l (m, a \cdot_B b') - l(m, b \cdot_B a') - n \cdot_N (a \cdot a')  \big) \stackrel{(\ref{bimod}), (\ref{bimod-l})}{=} 0.
\end{align*}
Thus, our claim holds. With all these notations, we are now in a position to construct the semidirect product relative Rota-Baxter algebra.

\begin{prop}\label{semi-rrb} {\bf (Semidirect product)} Let $M \xrightarrow{R} A$ be a relative Rota-Baxter algebra and $(N \xrightarrow{S} B, l, r)$ be a bimodule over it. Then $M \oplus N \xrightarrow{ R \oplus S} A \oplus B$ is a relative Rota-Baxter algebra.
\end{prop}

\begin{proof}
For any $(m,n), (m', n') \in M \oplus N$, we have
\begin{align*}
&(R \oplus S) (m,n) \cdot_\ltimes (R \oplus S) (m',n') \\
&= (R(m), S(n)) \cdot_\ltimes (R(m') , S(n')) \\
&= \big( R(m) \cdot R(m'), ~ R(m) \cdot_B S(n') + S(n) \cdot_B R(m')  \big) \\
&\stackrel{(\ref{bimod-rr}), (\ref{bimod-rs}), (\ref{bimod-sr})}{=} \big(  R \big( R(m) \cdot_M m' + m \cdot_M R(m')   \big),~ S\big( R(m) \cdot_N n' + l (m, S(n')) \big) + S \big( r (S(n), m') + n \cdot_N R(m')   \big)   \big)\\
&= (R \oplus S) \big( R(m) \cdot_M m' + m \cdot_M R(m'), ~ R(m) \cdot_N n' + r (S(n), m') + l(m, S(n')) + n \cdot_N R(m')  \big) \\
&= (R \oplus S) \big(  (R(m), S(n)) \vartriangleright (m', n')  + (m,n) \vartriangleleft (R(m'), S(n')) \big) \\
&= (R \oplus S) \bigg( \big( (R \oplus S)(m,n)  \big) \vartriangleright (m',n') + (m,n) \vartriangleleft \big(   (R \oplus S)(m',n') \big)  \bigg).
\end{align*}
This shows that $R \oplus S : M \oplus N \rightarrow A \oplus B$ is a relative Rota-Baxter operator. Hence the result follows.
\end{proof}

\medskip

\noindent {\bf Relations with dendriform algebras.} Here, we show that a bimodule over a relative Rota-Baxter algebra gives rise to a representation of the induced dendriform algebra. Furthermore, we find a converse of this result. In the end, we discuss the relationship between the cohomology (with coefficients) of a relative Rota-Baxter operator and the cohomology (with coefficients) of the induced dendriform algebra.

The proof of the following result is straightforward, hence we omit the details.

\begin{prop}\label{induced-dend-repn}
Let $M \xrightarrow{R} A$ be a relative Rota-Baxter algebra and $(N \xrightarrow{S} B, l, r)$ be a bimodule over it. Then $N $ is a {representation of the induced} dendriform algebra $(M , \prec_R, \succ_R)$ with action maps given by
\begin{align*}
m \prec n =  l (m, S(n)), \quad m \succ n = R(m) \cdot_N n, \quad n \prec m =  n \cdot_N R(m) ~~~ \text{ and } ~~~ n \succ m = r (S(n), m).
\end{align*}
\end{prop}

\begin{remark}
Let  $M \xrightarrow{R} A$ be a relative Rota-Baxter algebra and consider the adjoint bimodule $M \xrightarrow{R} A$ given in Example \ref{adj-bim}. Then the representation of the dendriform algebra $(M, \prec_R, \succ_R)$ on the vector space $M$ given in the above proposition coincides with the adjoint representation.
\end{remark}

\medskip

In Proposition \ref{induced-dend-repn}, we observed that a  bimodule over a relative Rota-Baxter algebra gives rise to a representation of the induced dendriform algebra. In the following, we will discuss the converse of this result. Let $(D, \prec, \succ)$ be a dendriform algebra. Consider the total associative algebra $D_{\mathrm{Tot}}$. Then there is a $D_{\mathrm{Tot}}$-bimodule structure on $D$ with left and right actions
\begin{align*}
\mathsf{x} \cdot x := \mathsf{x} \succ x ~~~~ \text{ and } ~~~~ x \cdot \mathsf{x} := x \prec \mathsf{x}, ~ \text{ for } \mathsf{x} \in D_{\mathrm{Tot}}, x \in D.
\end{align*}
With this bimodule, it is easy to verify that $D \xrightarrow{\mathrm{id}_D} D_\mathrm{Tot}$ is a relative Rota-Baxter algebra. Moreover, the induced dendriform structure on $D$ coincides with the given one.

Next, we take a representation $E$ of the dendriform algebra $(D, \prec, \succ)$. Then there are two $D_\mathrm{Tot}$-bimodule structures on $E$. The first one is given by
\begin{align*}
\mathsf{x} \cdot^1_E e := \mathsf{x} \prec e + \mathsf{x} \succ e ~~~ \text{ and } ~~~ e \cdot^1_E \mathsf{x} = e \prec \mathsf{x} + e \succ \mathsf{x},
\end{align*}
the second one is given by
\begin{align*}
\mathsf{x} \cdot^2_E e = \mathsf{x} \succ e ~~~ \text{ and } ~~~ e \cdot^2_E \mathsf{x} = e \prec \mathsf{x},
\end{align*}
for $\mathsf{x} \in D_\mathrm{Tot}$ and $e \in E$.
We denote the first $D_\mathrm{Tot}$-bimodule by $E_\mathrm{Tot}$ and the second one simply by $E$. Then it can be checked that $(E \xrightarrow{\mathrm{id}_E} E_\mathrm{Tot}, l, r)$ is a bimodule over the relative Rota-Baxter algebra $D \xrightarrow{\mathrm{id}_D} D_\mathrm{Tot}$, where the bilinear maps $l : D \otimes E_\mathrm{Tot} \rightarrow E$ and $r: E_\mathrm{Tot} \otimes D \rightarrow E$  are given by 
\begin{align*}
l (x, {e}) = x \prec {e} ~~~ \text{ and } ~~~ r({e}, x) = {e} \succ x, \text{ for } x \in D, {e} \in E_\mathrm{Tot}.
\end{align*}
Moreover, the induced {representation} of the dendriform algebra $D$ on the vector space $E$ coincides with the given one.

\begin{remark}\label{rmk-lift}
By lifting the relative Rota-Baxter algebra $D \xrightarrow{\mathrm{id}_D} D_\mathrm{Tot}$ and its bimodule $(E \xrightarrow{\mathrm{id}_E} E_\mathrm{Tot}, l, r)$, we obtain the Rota-Baxter algebra $(D_\mathrm{Tot} \oplus D, \widehat{\mathrm{id}_{D}})$ and its Rota-Baxter bimodule $(E_\mathrm{Tot} \oplus E, \widehat{\mathrm{id}_{E}})$. See Proposition \ref{prop-lift} for details. It is easy to see that the inclusion $D \hookrightarrow D_\mathrm{Tot} \oplus D$ is an embedding of the dendriform algebra $D$ into the Rota-Baxter algebra $(D_\mathrm{Tot} \oplus D, \widehat{\mathrm{id}_{D}})$. In other words, the inclusion map $D \hookrightarrow D_\mathrm{Tot} \oplus D$ is a morphism of dendriform algebras, where $D_\mathrm{Tot} \oplus D$ is equipped with the dendriform algebra induced from the Rota-Baxter operator $\widehat{\mathrm{id}_{D}}$. Similarly, the inclusion $E \hookrightarrow E_\mathrm{Tot} \oplus E$ is an embedding of the {representation} $E$ (of the dendriform algebra $D$) into the Rota-Baxter bimodule $(E_\mathrm{Tot} \oplus E, \widehat{\mathrm{id}_{E}})$ of the Rota-Baxter algebra $(D_\mathrm{Tot} \oplus D, \widehat{\mathrm{id}_{D}})$.
\end{remark}

\medskip

In the following, we find some cohomological relations with dendriform algebras. 
Let $M \xrightarrow{R} A$ be a relative Rota-Baxter algebra. Then we have seen in Proposition \ref{prop-rrb-dend} that $M_\mathrm{Tot} = (M, \ast_R)$ is an associative algebra, where
\begin{align*}
m \ast_R m' = R(m ) \cdot_M m' + m \cdot_M R(m'), ~ \text{ for } m, m' \in M. 
\end{align*}
Let $(N \xrightarrow{S} B, l, r)$ be a bimodule over the relative Rota-Baxter algebra $M \xrightarrow{R} A$. We define bilinear maps $\triangleright_B : M \otimes B \rightarrow B$ and $\triangleleft_B ~: B \otimes M \rightarrow B$ by
\begin{align*}
m \triangleright_B b = R(m ) \cdot_B b - S (l(m,b)) ~~~ \text{and } ~~~ b \triangleleft_B m = b \cdot_B R(m) - S( r (b,m)), ~ \text{ for } m \in M,~ b \in B.
\end{align*}
Then we have the following.

\begin{prop}\label{denote-bimod}
With the above notations, $B$ is a bimodule over the associative algebra $M_\mathrm{Tot}$. (We denote this bimodule by $B_{\triangleright, \triangleleft}$).
\end{prop}

\begin{proof}
For any $m, m' \in M$ and $b \in B$, we have
\begin{align*}
&( m \ast_R m' ) \triangleright_B b - m \triangleright_B (m' \triangleright_B b) \\
&= R (m \ast_R m') \cdot_B b - S (l (m \ast_R m', b)) - R(m) \cdot_B (m ' \triangleright_B b)  + S (l (m, m' \triangleright_B b)) \\
&= {(R(m) \cdot R(m')) \cdot_B b}  - S \big( l ( R(m) \cdot_M m', b) \big) - {S \big( l ( m \cdot_M R(m'), b) \big)}  \\
& ~~~ - {R(m) \cdot_B (R(m') \cdot_B b)} + R(m) \cdot_B S (l (m', b))
+ {S \big(l (m, R(m') \cdot_B b) \big)} - S \big( l (m, S \circ l (m', b))  \big) \\
&=  - S \big( l ( R(m) \cdot_M m', b) \big)  + S \big(  R(m) \cdot_M l(m', b) + l (m, S \circ l (m', b))    \big)  - S \big( l (m, S \circ l (m', b))  \big)  = 0.
\end{align*}
Similarly,
\begin{align*}
&(m \triangleright_B b) \triangleleft_B m' - m \triangleright_B (b \triangleleft_B m') \\
&= (m \triangleright_B b) \cdot_B R(m') - S \big( r (m \triangleright_B b, m')\big) - R(m) \cdot_B (b \triangleleft_B m') + S \big( l (m, b \triangleleft_B m')  \big) \\
&= {(R(m) \cdot_B b) \cdot_B R(m')} - S (l(m, b)) \cdot_B R(m') - S \big( r (R(m) \cdot_B b, m')  \big) + S \big( r (S \circ l (m, b), m')  \big) \\
& ~~~ - {R(m) \cdot_B (b \cdot_B R(m'))} + R(m) \cdot_B S(r (b, m')) + S \big(  l (m, b \cdot_B R(m'))  \big) - S \big(  l (m, S \circ r (b,m'))  \big) \\
&= - S \big(  r (  S \circ l (m,b), m') ~+~ l(m,b) \cdot_N R(m')  \big) - S \big( r (R(m) \cdot_B b, m')  \big) + S \big( r (S \circ l (m, b), m')  \big) \\
& ~~~ + S \big( R(m) \cdot_N r (b, m') ~+~ l (m, S \circ r (b, m'))  \big) + S \big(  l (m, b \cdot_B R(m'))  \big) - S \big(  l (m, S \circ r (b,m'))  \big) = 0
\end{align*}
and
\begin{align*}
&(b \triangleleft_B m) \triangleleft_B m' - b \triangleleft_B (m \ast_R m') \\
& = (b \triangleleft_B m) \cdot_B R(m') - S \big( r (b \triangleleft_B m, m') \big) - b \cdot_B R (m \ast_R m') + S \big( r (b, m \ast_R m') \big)  \\
& = {(b \cdot_B R(m)) \cdot_B R(m')} - S (r(b,m)) \cdot_B R(m') - {S \big( r (b \cdot_B R(m), m')  \big)} + S \big( r (S \circ r (b,m), m')      \big) \\
& ~~~ - {b \cdot_B (R(m) \cdot R(m'))} + {S \big( r (b, R(m) \cdot_M m')  \big)} + S \big( r (b, m \cdot_M R(m'))   \big) \\
&= - S \big(  r ( S \circ r (b,m), m') + r(b,m) \cdot_N R(m')   \big)  + S \big( r (S \circ r (b,m), m')      \big) + S \big( r (b, m \cdot_M R(m'))   \big) = 0.
\end{align*}
This proves the result.
\end{proof}

It follows from the above proposition that we may consider the Hochschild cohomology of the associative algebra $M_\mathrm{Tot}$ with coefficients in the bimodule $B_{\triangleright, \triangleleft}$. More precisely, we define the cochain complex $\{ C^\bullet_\mathsf{H} (M_\mathrm{Tot}, B_{\triangleright, \triangleleft}), \delta_{M,B} \}$, where $C^k_\mathsf{H} (M_\mathrm{Tot}, B_{\triangleright, \triangleleft}) = \mathrm{Hom} (M^{\otimes k}, B)$, for $k \geq 0$, and $\delta_{M,B} : C^k_\mathsf{H} (M_\mathrm{Tot}, B_{\triangleright, \triangleleft}) \rightarrow C^{k+1}_\mathsf{H} (M_\mathrm{Tot}, B_{\triangleright, \triangleleft})$ given by
\begin{align*}
&(\delta_{M,B} f)(m_1,\ldots,m_{k+1})\\
&=m_1 \triangleright_B f(m_2,\ldots,m_{k+1})+\sum_{i=1}^k (-1)^i~ f(m_1,\ldots,m_i \ast_R m_{i+1},\ldots,m_{k+1})+(-1)^{k+1} ~f(m_1,\ldots,m_{k}) \triangleleft_B m_{k+1}\\
&=R(m_1)\cdot_B f(m_2,\ldots,m_{k+1})-S(l(m_1,f(m_2,\ldots,m_{k+1})))\\
&~~+\sum_{i=1}^k (-1)^i f(m_1,\ldots,R(m_i)\cdot_M m_{i+1},\ldots,m_{k+1})+\sum_{i=1}^k (-1)^i~ f(m_1,\ldots,m_i\cdot_M R(m_{i+1}),\ldots,m_{k+1})\\
&~~+(-1)^{k+1}\big(f(m_1,\ldots,m_{k})\cdot_B R(m_{k+1})-S(r(f(m_1,\ldots,m_{k}),m_{k+1})\big).
\end{align*}
We denote the corresponding cohomology groups by $H^\bullet_\mathsf{H} (M_\mathrm{Tot}, B_{\triangleright, \triangleleft})$.

\begin{remark}\label{rem-fin}
Let $M \xrightarrow{R} A$ be a relative Rota-Baxter algebra and consider the adjoint bimodule $M \xrightarrow{R} A$ (see Example \ref{adj-bim}). Then it follows from the previous discussion that the vector space $A$ carries a bimodule structure over the associative algebra $M_\mathrm{Tot}$. This bimodule has considered first by Uchino \cite{uchino} and further studied in \cite{das-rota}. The corresponding Hochschild cohomology groups are called the cohomology of the relative Rota-Baxter operator $R$.
\end{remark}

Given a relative Rota-Baxter algebra $M \xrightarrow{R} A$, in \cite{das-rota} the author finds a connection between the cohomology of the relative Rota-Baxter operator $R$ and the cohomology of the induced dendriform algebra $(M, \prec_R, \succ_R)$ with coefficients in the adjoint representation. We will extend this result in the context of relative Rota-Baxter algebras equipped with bimodules.

We first recall the cohomology of a dendriform algebra with coefficients in a representation \cite{das-dend}. Let $(D, \prec, \succ)$ be a dendriform algebra. Consider the associative algebra $D_\mathrm{Tot} \oplus D$  given in Remark \ref{rmk-lift}. Note that the associative multiplication on $D_\mathrm{Tot} \oplus D$  is given by
\begin{align*}
(\mathsf{x} , x) \cdot (\mathsf{y}, y) = (\mathsf{x} \ast \mathsf{y},~ \mathsf{x} \succ y + x \prec \mathsf{y}), \text{ for } (\mathsf{x} , x), (\mathsf{y}, y) \in D_\mathrm{Tot} \oplus D.
\end{align*}
If $E$ is a representation of the dendriform algebra $D$, then $E_\mathrm{Tot} \oplus E$ is a bimodule over the associative algebra $D_\mathrm{Tot} \oplus D$ with left and right actions given by 
\begin{align*}
(\mathsf{x}, x) \cdot (\mathsf{e}, e) = (\mathsf{x} \prec \mathsf{e} + \mathsf{x} \succ \mathsf{e},~ \mathsf{x} \succ e + x \prec \mathsf{e})
 ~\text{ and }~ (\mathsf{e}, e) \cdot (\mathsf{x}, x) = (\mathsf{e} \prec \mathsf{x} + \mathsf{e} \succ \mathsf{x},~ \mathsf{e} \succ x + e \prec \mathsf{x}),
\end{align*}
for $(\mathsf{x}, x) \in D_\mathrm{Tot} \oplus D$ and $(\mathsf{e}, e) \in E_\mathrm{Tot} \oplus E$.
Therefore, we may define the Hochschild cochain complex $\{ C^\bullet_\mathsf{H} ( D_\mathrm{Tot} \oplus D, E_\mathrm{Tot} \oplus E ), \delta_{D_\mathrm{Tot} \oplus D, E_\mathrm{Tot} \oplus E} \}$. Using this complex, we will now define the cochain complex of the dendriform algebra $D$ with coefficients in the representation $E$. We first need the following notations. For each $k \geq 1$, let $C_k$ be the set of first $k$ natural numbers. For convenience, we write $C_k = \{ [1], [2], \ldots, [k] \}$. We define the $k$-th cochain group as
\begin{align*}
C^k_\mathsf{D} (D, E) := \mathrm{Hom}({\bf k}[C_k] \otimes D^{\otimes k}, E), ~ \text{ for } k \geq 1.
\end{align*} 
For any $f \in C^k_\mathsf{D} (D, E)$, there is an element $\widehat{f} \in C^k_\mathsf{H} ( D_\mathrm{Tot} \oplus D, E_\mathrm{Tot} \oplus E )$ given by
\begin{align*}
\widehat{f} \big( ({x}_1, 0), \ldots, ({x}_k, 0)  \big) =~& \big( \sum_{i=1}^k f([i]; x_1, \ldots, x_k), 0    \big),\\
\widehat{f} \big( (x_1, 0), \ldots, (0, x_i), \ldots, (x_k, 0)   \big) =~& \big(   0, f([i]; x_1, \ldots, x_k) \big), \\
\widehat{f} \big(  (x_1, 0), \ldots, (0, x_i), \ldots, (0, x_j), \ldots, (x_k, 0)  \big) =~& (0,0).
\end{align*}
Note that $f$ can be obtained from $\widehat{f}$ by the following
\begin{align*}
f([i]; x_1, \ldots, x_k) = \mathrm{pr}_2 \circ \widehat{f} \big( (x_1, 0), \ldots, (0, x_i), \ldots, (x_k, 0)   \big), \text{ for } [i] \in C_k,
\end{align*}
where $\mathrm{pr}_2 : E_\mathrm{Tot} \oplus E \rightarrow E$ is the projection onto the second factor. We define a map $\delta_\mathsf{D} : C^k_\mathsf{D} (D, E) \rightarrow C^{k+1}_\mathsf{D} (D, E)$ implicitly by the following formula
\begin{align*}
\widehat{\delta_\mathsf{D} (f)} = \delta_{D_\mathrm{Tot} \oplus D, E_\mathrm{Tot} \oplus E} (\widehat{f}), ~ \text{ for } f \in C^k_\mathsf{D} (D, E).
\end{align*}
The explicit formula of the differential $\delta_\mathsf{D}$ can be found in \cite{das-dend}. The cohomology groups of the cochain complex $\{ C^\bullet_\mathsf{D}(D, E), \delta_\mathsf{D} \}$ are called the cohomology of the dendriform algebra $D$ with coefficients in the representation $E$. 

\medskip

Let $M \xrightarrow{R} A$ be a relative Rota-Baxter algebra and $(N \xrightarrow{S} B, l, r)$ be a bimodule over it. We consider the cochain complex $\{  C^\bullet_\mathsf{H} (M_\mathrm{Tot}, B_{\triangleright, \triangleleft}), \delta_{M,B} \}$. On the other hand, we know from Proposition \ref{induced-dend-repn} that $N$ is a representation of the dendriform algebra $(M, \prec_R, \succ_R)$. Hence we may consider the dendriform cochain complex $\{ C^\bullet_\mathsf{D} (M, N), \delta_\mathsf{D} \}$.

For each $k \geq 1$, we define a map $\psi_k : C^k_\mathsf{H} (M_\mathrm{Tot}, B_{\triangleright, \triangleleft}) \rightarrow C^{k+1}_\mathsf{D} (M, N)$ by
\begin{align*}
\psi_k (f) ([r]; m_1, \ldots, m_{k+1}) = \begin{cases}
(-1)^{k+1}~ l (m_1, f (m_2, \ldots, m_{k+1})) & \text{ if }~ r = 1 \\
0 & \text{ if }~ 2 \leq r \leq k\\
r (f (m_1, \ldots, m_k), m_{k+1}) & \text{ if }~ r = k+1.
\end{cases}
\end{align*}

\begin{prop}\label{coho-coho-rel}
With the above notations, we have
\begin{align*}
\delta_\mathsf{D} \circ \psi_k = \psi_{k+1} \circ \delta_{M,B}.
\end{align*} 
Hence the collection $\{\psi_k\}_{k \geq 1}$ of maps induces a map $\psi_\bullet : H^\bullet_\mathsf{H} (M_\mathrm{Tot}, B_{\triangleright, \triangleleft}) \rightarrow H^{\bullet + 1 }_\mathsf{D} (M, N)$ on cohomology.
\end{prop}

\begin{proof}
We first consider the semidirect product relative Rota-Baxter algebra $M \oplus N \xrightarrow{R \oplus S} A \oplus B$ given in Proposition \ref{semi-rrb}. Hence it follows from Proposition \ref{prop-rrb-dend} that the space $M \oplus N$ carries a dendriform algebra structure with products
\begin{align*}
(m,n) \prec_{R \oplus S} (m', n') =~&  (m \cdot_M R (m'), ~ l (m, S (n') + n \cdot_N R (m')),\\
(m,n) \succ_{R \oplus S} (m', n') =~& (R(m) \cdot_M m', ~ R(m) \cdot_N n' + r (S(n), m')).
\end{align*}
The corresponding total associative multiplication is given by
\begin{align*}
(m,n) \ast_{R \oplus S} (m',n') = (m \ast_R m', ~  l (m, S (n') + n \cdot_N R (m') + R(m) \cdot_N n' + r (S(n), m')).
\end{align*}
Moreover, $A \oplus B$ is a bimodule over this associative algebra with left and right actions
\begin{align*}
(m,n) \triangleright_{A \oplus B} (a, b) =~& \big( R(m) \cdot a - R (m \cdot_M a),~ R(m) \cdot_B  b + S(n) \cdot_B a - S ( l (m, b) + n \cdot_N a)  \big),\\
(a,b) \triangleleft_{A \oplus B} (m,n) =~& \big(  a \cdot R(m) - R (a \cdot_M m), ~ a \cdot_B S(n) + b \cdot_B R(m) - S ( a \cdot_N n + r (b, m))  \big).
\end{align*}
It follows from \cite[Proposition 3.5]{das-rota} that the collection $\{\Psi_k\}_{k \geq 1}$ of maps
\begin{align*}
& \Psi_k : C^k_\mathsf{H} (M \oplus N, A \oplus B) \rightarrow C^{k+1}_\mathsf{D} (M \oplus N, M \oplus N) \\
\Psi_k (F)([r]; (m_1, n_1),& \ldots, (m_{k+1}, n_{k+1})) = \begin{cases}
(-1)^{k+1} ~ (m_1, n_1) \triangleleft F ( (m_2, n_2), \ldots, (m_{k+1}, n_{k+1})) & \text{if } r = 1 \\
0 & \text{if } 2 \leq r \leq k \\
F ( (m_1, n_1), \ldots, (m_{k}, n_{k})) \triangleright (m_{k+1}, n_{k+1}) & \text{if } r = k+1
\end{cases}
\end{align*}
defines a morphism of cochain complexes from $\{ C^\bullet_\mathsf{H} (M \oplus N, A \oplus B), \delta_{ M \oplus N, A \oplus B } \}$ to $\{ C^{\bullet + 1}_\mathsf{D} (M \oplus N, M \oplus N), \delta_\mathsf{D} \}$. Note that $\{  C^\bullet_\mathsf{H} (M_\mathrm{Tot}, B_{\triangleright, \triangleleft}), \delta_{M,B} \}$ is a subcomplex of  $\{ C^\bullet_\mathsf{H} (M \oplus N, A \oplus B), \delta_{M\oplus N, A \oplus B} \}$, and $\{C^{\bullet +1}_\mathsf{D} (M,N), \delta_\mathsf{D} \}$ is a subcomplex of  $\{ C^{\bullet + 1}_\mathsf{D} (M \oplus N, M \oplus N), \delta_\mathsf{D} \}$. Finally, the result follows as the restriction of $\Psi_k$ to the subspace $C^k_\mathsf{H} (M_\mathrm{Tot}, B_{\triangleright, \triangleleft})$ is given by $\psi_k$, for all $k \geq 1$. This proves the result.
\end{proof}

\section{Cohomology of relative Rota-Baxter algebras with coefficients in bimodule}\label{sec-4}

In this section, we define the cohomology of a relative Rota-Baxter algebra with coefficients in a bimodule. This cohomology generalizes the cohomology of relative Rota-Baxter algebras (with coefficients in the adjoint bimodule) defined in \cite{DasSK}.


We first recall the cochain complex of a relative Rota-Baxter algebra $M \xrightarrow{R} A$. Let $\mathcal{A}^{k-1, 1}$ be the direct sum of all possible $k$ tensor powers of $A$ and $M$ in which $A$ appears $k-1$ times (hence $M$ appears exactly once). For instance,
\begin{align*}
\mathcal{A}^{0,1} = M, \quad \mathcal{A}^{1,1} = ( A \otimes M) \oplus (M \otimes A), \quad \mathcal{A}^{2,1} = (A \otimes A \otimes M) \oplus (A \otimes M \otimes A) \oplus (M \otimes A \otimes A).
\end{align*}
For each $k\geq 0$, the space of $k$-cochains $C^k(A,M,R)$ is given by 
\begin{align*}
C^k(A,M,R) = \begin{cases}
0 & \text{ if } k = 0, \\
\mathrm{Hom}(A, A) \oplus \mathrm{Hom}(M, M)& \text{ if } k = 1, \\ 
\mathrm{Hom}(A^{\otimes k}, A) \oplus \mathrm{Hom}(\mathcal{A}^{k-1,1}, M) \oplus \mathrm{Hom}(M^{\otimes k-1}, A) & \text{ if } k \geq 2.
\end{cases}
\end{align*}
The coboundary operator $\mathcal{D}:C^k(A,M,R)\rightarrow C^{k+1}(A,M,R)$ is given by
\begin{align*}
\mathcal{D}(\alpha,\beta)=~& (\delta_A(\alpha),\delta^\alpha_{A,M}(\beta),h_R(\alpha,\beta)), ~ \mbox{for }(\alpha,\beta)\in C^1(A,M,R) = \mathrm{Hom}(A, A) \oplus \mathrm{Hom}(M, M),\\
\mathcal{D}(\alpha,\beta,\gamma)=~& (\delta_A(\alpha),\delta^\alpha_{A,M}(\beta),\delta_{M,A}(\gamma)+h_R(\alpha,\beta)), ~ \mbox{for }(\alpha,\beta, \gamma)\in C^{k \geq 2} (A,M,R).
\end{align*}
Here $\delta_A$ is the Hochschild coboundary operator of the associative algebra $A$ with coefficients in the adjoint bimodule, and for any $\alpha \in \mathrm{Hom}(A^{\otimes k}, A)$, the map $\delta_{A, M}^\alpha : \mathrm{Hom} (\mathcal{A}^{k-1,1}, M) \rightarrow \mathrm{Hom} (\mathcal{A}^{k,1}, M)$ is given by
\begin{align*}
\delta_{A, M}^\alpha (\beta) (a_1, \ldots, a_{k+1}) =&~ (\mu+ l_M) (a_1, (\alpha + \beta) (a_2, \ldots, a_{k+1})) \\
&+ \sum_{i=1}^k (-1)^i ~\beta (a_1, \ldots, (\mu + l_M + r_M)(a_i, a_{i+1}), \ldots, a_{k+1}) \\
&+ (-1)^{k+1} ~ (\mu + r_M) ((\alpha + \beta) (a_1, \ldots, a_k), a_{k+1} ),
\end{align*}
for $\beta \in \mathrm{Hom}(\mathcal{A}^{k-1,1}, M)$ and $a_1, \ldots, \widehat{a_s}, \ldots, a_{k+1} \in A$, $a_s \in M$ ($1 \leq s \leq k+1$). The map $\delta_{M,A}$ is the Hochschild coboundary operator of the associative algebra $M_\mathrm{Tot}$ with coefficients in the $M_\mathrm{Tot}$-bimodule $A$ considered in Remark \ref{rem-fin}, and the map $h_R : \mathrm{Hom} (A^{\otimes k},A) \oplus \mathrm{Hom} (\mathcal{A}^{k-1,1}, M) \rightarrow \mathrm{Hom}(M^{\otimes k}, A)$ is given by
\begin{align*}
h_R (\alpha, \beta) (m_1, \ldots, m_k ) = (-1)^k \big\{ \alpha (R(m_1), \ldots, R(m_k)) - \sum_{i=1}^k R \circ \beta \big( R(m_1), \ldots, m_i, \ldots, R(m_k)  \big)  \big\}.
\end{align*} 
It has been shown in \cite{DasSK} that $\{ C^\bullet (A, M, R), \mathcal{D} \}$ is a cochain complex. The corresponding cohomology groups are denoted by $H^\bullet(A,M,R)$, and called the cohomology of the relative Rota-Baxter algebra $M \xrightarrow{R} A$. 

In the same reference, the authors also considered formal deformations of a relative Rota-Baxter algebra $M \xrightarrow{R} A$ and showed that the above cohomology govern such deformations.


\medskip

\medskip

\noindent {\bf Cohomology of a relative Rota-Baxter algebra with coefficients in a bimodule.} Here we define the cochain complex of a relative Rota-Baxter algebra with coefficients in a bimodule. Let $M \xrightarrow{R} A$ be a relative Rota-Baxter algebra and $(N \xrightarrow{S} B, l, r)$ be a bimodule over it. We fix the following notations.

- Let $\delta_{A, B} : \mathrm{Hom} (A^{\otimes k}, B) \rightarrow \mathrm{Hom} (A^{\otimes k+1}, B)$, for $k \geq 1$, be the Hochschild coboundary operator of the algebra $A$ with coefficients in the $A$-bimodule $B$.

- For any $\alpha \in \mathrm{Hom}(A^{\otimes k}, B)$ with $k \geq 1$, we define a map $\delta_{A, N}^\alpha : \mathrm{Hom} (\mathcal{A}^{k-1,1}, N) \rightarrow \mathrm{Hom} (\mathcal{A}^{k,1}, N)$ by
\begin{align*}
\delta_{A, N}^\alpha (\beta) (a_1, \ldots, a_{k+1}) =~& (l+ l_N) (a_1, (\alpha + \beta) (a_2, \ldots, a_{k+1})) \\
&+ \sum_{i=1}^k (-1)^i ~\beta (a_1, \ldots, (\mu + l_M + r_M)(a_i, a_{i+1}), \ldots, a_{k+1}) \\
&+ (-1)^{k+1} ~ (r + r_N) ((\alpha + \beta) (a_1, \ldots, a_k), a_{k+1} ),
\end{align*}
\qquad for $\beta \in \mathrm{Hom}(\mathcal{A}^{k-1,1}, N)$ and $a_1, \ldots, \widehat{a_s}, \ldots, a_{k+1} \in A$, $a_s \in M$ ($1 \leq s \leq k+1$).  

- Let $\delta_{M, B} : \mathrm{Hom} (M^{\otimes k-1}, B) \rightarrow \mathrm{Hom} (M^{\otimes k}, B)$, for $k \geq 2$, be the Hochschild coboundary operator of the algebra ${M_\mathrm{Tot}}$ with coefficients in the bimodule $B_{\triangleright, \triangleleft}$ given in Proposition \ref{denote-bimod}.

- Finally, we define a map $h_R : \mathrm{Hom} (A^{\otimes k}, B) \oplus \mathrm{Hom} (\mathcal{A}^{k-1,1}, N) \rightarrow \mathrm{Hom}(M^{\otimes k}, B)$ by
\begin{align*}
h_R (\alpha, \beta) (m_1, \ldots, m_k ) = (-1)^k \big\{ \alpha (R(m_1), \ldots, R(m_k)) - \sum_{i=1}^k S \circ \beta \big( R(m_1), \ldots, m_i, \ldots, R(m_k)  \big)  \big\}.
\end{align*}

\medskip

We are now in a position to define the cohomology of the relative Rota-Baxter algebra $M \xrightarrow{R} A$ with coefficients in the bimodule $(N \xrightarrow{S} B, l, r)$. For each $k \geq 0$, we define  $C^k_\mathsf{rRB} (M \xrightarrow{R} A; N \xrightarrow{S} B)$ by
\begin{align*}
C^k_\mathsf{rRB} (M \xrightarrow{R} A; N \xrightarrow{S} B) = \begin{cases}
0 & \text{ if } k = 0, \\
\mathrm{Hom}(A, B) \oplus \mathrm{Hom}(M, N)& \text{ if } k = 1, \\ 
\mathrm{Hom}(A^{\otimes k}, B) \oplus \mathrm{Hom}(\mathcal{A}^{k-1,1}, N) \oplus \mathrm{Hom}(M^{\otimes k-1}, B) & \text{ if } k \geq 2.
\end{cases}
\end{align*}
Define a map $\delta_\mathsf{rRB} : C^k_\mathsf{rRB} (M \xrightarrow{R} A; N \xrightarrow{S} B) \rightarrow C^{k+1}_\mathsf{rRB} (M \xrightarrow{R} A; N \xrightarrow{S} B)$ by
\begin{align*}
\delta_\mathsf{rRB} (\alpha,\beta) =~& \big(\delta_{A, B} (\alpha),~ \delta_{A, N}^\alpha (\beta),~ h_R (\alpha, \beta) \big), \text{ for } (\alpha,\beta) \in C^1_\mathsf{rRB} (M \xrightarrow{R} A; N \xrightarrow{S} B), \\
\delta_\mathsf{rRB} ((\alpha, \beta, \gamma)) =~& \big( \delta_{A, B} (\alpha),~ \delta_{A, N}^\alpha (\beta),~ \delta_{M, B} (\gamma) + h_R (\alpha, \beta)  \big), \text{ for } (\alpha, \beta, \gamma) \in C^{k \geq 2}_\mathsf{rRB} (M \xrightarrow{R} A; N \xrightarrow{S} B).
\end{align*}

\begin{prop}
With all these notations, we have  $(\delta_\mathsf{rRB} )^2 = 0$.
\end{prop}

\begin{proof}
If $M \xrightarrow{R} A$ is a relative Rota-Baxter algebra and $(N \xrightarrow{S} B, l, r)$ is a bimodule over it, then by Proposition \ref{semi-rrb} we have the semidirect product relative Rota-Baxter algebra  $M \oplus N \xrightarrow{ R \oplus S} A \oplus B$. Let $\{C^\bullet(A\oplus B,M\oplus N,R\oplus S),\mathcal{D}\}$ be the cochain complex of this semidirect product relative Rota-Baxter algebra (with coefficients in the adjoint bimodule). Then there is an obvious inclusion
\begin{align*}
C^k_\mathrm{rRB} (M \xrightarrow{R} A ; N \xrightarrow{S} B) &~\hookrightarrow ~ C^k (A \oplus B, M \oplus N, R\oplus S), \text{ for } k \geq 1.
\end{align*}
It can be easily verified that the map $\delta_\mathsf{rRB} $ is the restriction of the map $\mathcal{D}$ on the collection of subspaces $C^\bullet_\mathrm{rRB} (M \xrightarrow{R} A ; N \xrightarrow{S} B)$. In other words, $\{ C^\bullet_\mathrm{rRB} (M \xrightarrow{R} A ; N \xrightarrow{S} B), \delta_\mathrm{rRB}\}$ is a subcomplex of the cochain complex $\{C^\bullet(A\oplus B,M\oplus N,R\oplus S),\mathcal{D}\}$, which implies that $(\delta_\mathrm{rRB})^2 = 0$.
\end{proof}

\medskip

It follows from the above proposition that $\{ C^\bullet_\mathsf{rRB} (M \xrightarrow{R} A; N \xrightarrow{S} B), \delta_\mathsf{rRB} \}$ is a cochain complex. Let $Z^k_\mathsf{rRB} (M \xrightarrow{R} A; N \xrightarrow{S} B)$ and $B^k_\mathsf{rRB} (M \xrightarrow{R} A; N \xrightarrow{S} B)$ be the space of $k$-cocycles and $k$-coboundaries, respectively. The corresponding quotients
\begin{align*}
H^k_\mathsf{rRB} (M \xrightarrow{R} A; N \xrightarrow{S} B) := \frac{  Z^k_\mathsf{rRB} (M \xrightarrow{R} A; N \xrightarrow{S} B)}{ B^k_\mathsf{rRB} (M \xrightarrow{R} A; N \xrightarrow{S} B)}, \text{ for } k \geq 0
\end{align*}
are called the cohomology of the relative Rota-Baxter algebra $M \xrightarrow{R} A$ with coefficients in the bimodule $(N \xrightarrow{S} B, l, r).$


\medskip

Let $M \xrightarrow{R} A$ be a relative Rota-Baxter algebra and $(N \xrightarrow{S} B, l, r)$ be a bimodule over it. A pair $(\alpha, \beta) \in \mathrm{Hom}(A, B) \oplus \mathrm{Hom}(M,N) = C^1_\mathsf{rRB} (M \xrightarrow{R} A; N \xrightarrow{S} B)$ is said to be a derivation on the relative Rota-Baxter algebra $M \xrightarrow{R} A$ with values in the bimodule $(N \xrightarrow{S} B, l, r)$ if the following identities are hold:
\begin{align*}
\alpha (a \cdot a') =~& \alpha (a) \cdot_B a' + a \cdot_B \alpha (a'),\\
\beta (a \cdot_M m) =~& r (\alpha (a), m) + a \cdot_N \beta (m),\\
\beta (m \cdot_M a) =~& \beta (m) \cdot_N a + l (m, \alpha (a)),\\
\alpha (R(m)) =~& S (\beta (m)), 
\end{align*} 
for $a, a' \in A$ and $m \in M$. The set of all derivations are denoted by $\mathrm{Der} (M \xrightarrow{R} A; N \xrightarrow{S} B)$. 
It follows from the definition that $$Z^1_\mathsf{rRB}(M \xrightarrow{R} A; N \xrightarrow{S} B) = \mathrm{Der} (M \xrightarrow{R} A; N \xrightarrow{S} B).$$


\medskip

\begin{remark}
Let $M \xrightarrow{R} A$ be a relative Rota-Baxter algebra and consider it as a adjoint bimodule over itself (Example \ref{adj-bim}). Then the cochain complex $\{ C^\bullet ( M \xrightarrow{R} A ; M \xrightarrow{R} A), \delta_\mathrm{rRB} \}$ coincides with the complex $\{ C^\bullet (A, M, R), \mathcal{D} \}$ given in \cite{DasSK}. Hence the corresponding cohomology groups are the same.
\end{remark}


\noindent {\bf Cohomology of a Rota-Baxter algebra with coefficients in a Rota-Baxter bimodule.} Let $(A,R)$ be a Rota-Baxter algebra and $(M, R_M)$ be a Rota-Baxter bimodule. In Example \ref{rb-bim}, we observed that the Rota-Baxter algebra $(A,R)$ can be considered as a relative Rota-Baxter algebra $A \xrightarrow{R} A$ and the Rota-Baxter bimodule $(M, R_M)$ can be considered as a bimodule $M \xrightarrow{R_M} M$ over the relative Rota-Baxter algebra $A \xrightarrow{R} A$. Consider the cochain complex $\{ C^\bullet_\mathsf{rRB} (A \xrightarrow{R} A; M \xrightarrow{R_M} M), \delta_\mathsf{rRB} \}$ of the relative Rota-Baxter algebra $A \xrightarrow{R} A$ with coefficients in the bimodule $M \xrightarrow{R_M} M$. 

For each $k \geq 0$, we define $C^k_\mathsf{RB}(A,R; M, R_M)$ by
\begin{align*}
C^k_\mathsf{RB}(A, R;M, R_M) = \begin{cases} 0 & \text{ if } k =0,\\
\mathrm{Hom}(A, M)&\text{ if } k =1, \\
\mathrm{Hom}(A^{\otimes k}, M) \oplus \mathrm{Hom}(A^{\otimes k-1}, M) & \text{ if } k \geq 2.
\end{cases}
\end{align*}
Then there is an embedding $i : C^k_\mathsf{RB} (A,R;M, R_M) \rightarrow  C^k_\mathsf{rRB} (A \xrightarrow{R} A; M \xrightarrow{R_M} M)$ given by
\begin{align*}
i(\alpha) = (\alpha,\alpha) ~~~~ \text{ and } ~~~~ i(\beta, \gamma) = (\beta, \beta, \gamma), ~ \text{ for } \alpha \in \mathrm{Hom}(A,M) \text{ and } (\beta, \gamma) \in C^{k \geq 2}_\mathsf{RB}(A,R;M, R_M).
\end{align*}
It is easy to see that $\delta_\mathsf{rRB} (\mathrm{im~} (i)) \subset \mathrm{im~} (i)$. Hence the differential $\delta_\mathsf{rRB}$ restricts to a differential (which we denote by $\delta_\mathsf{RB}$) on $C^\bullet_\mathsf{RB}(A,R;M, R_M)$. 
In other words, the complex $\{ C^\bullet_\mathsf{RB} (A, R; M, R_M), \delta_\mathsf{RB} \}$ is a subcomplex of $\{ C^\bullet_\mathsf{rRB} (A \xrightarrow{R} A; M \xrightarrow{R_M} M), \delta_\mathsf{rRB} \}$. The cohomology groups of $\{ C^\bullet_\mathsf{RB} (A, R; M, R_M), \delta_\mathsf{RB} \}$ are called the cohomology of the Rota-Baxter algebra $(A, R)$ with coefficients in the Rota-Baxter bimodule $(M, R_M)$. Note that the above cochain complex coincides with the one given in \cite{wang-zhou}. Hence the cohomologies are also same.





\section{Abelian extensions of relative Rota-Baxter algebras}\label{sec-5}

In this section, we study abelian extensions of a relative Rota-Baxter algebra by a bimodule. Our main result classifies isomorphism classes of such abelian extensions by the second cohomology group of the relative Rota-Baxter algebra.

Let $M \xrightarrow{R} A$ be a relative Rota-Baxter algebra and $N \xrightarrow{S} B$ be a $2$-term chain complex. Note that $N \xrightarrow{S} B$ can be regarded as a relative Rota-Baxter algebra with trivial associative multiplication on $B$ and its trivial bimodule structure on $N$.

\begin{defn}
An {\bf abelian extension} of a relative Rota-Baxter algebra $M \xrightarrow{R} A$ by a $2$-term cochain complex $N \xrightarrow{S} B$ is a short exact sequence of relative Rota-Baxter algebras
\begin{align}\label{abelian-diagram}
\xymatrix{
o \ar[r] & N \ar[r]^{\overline{i}} \ar[d]_S & \widehat{M} \ar[r]^{\overline{p}} \ar[d]^{\widehat{R}} & M \ar[r] \ar[d]^R & 0 \\
0 \ar[r] & B \ar[r]_i & \widehat{A}  \ar[r]_p & A \ar[r] & 0.
}
\end{align} 
\end{defn}

Let (\ref{abelian-diagram}) be an abelian extension of the relative Rota-Baxter algebra $M \xrightarrow{R} A$ by a $2$-term cochain complex $N \xrightarrow{S} B$. A section of (\ref{abelian-diagram}) is a pair $(s, \overline{s})$ of linear maps $s : A \rightarrow \widehat{A}$ and $\overline{s} : M \rightarrow \widehat{M}$ satisfying $p \circ s = \mathrm{id}_A$ and $\overline{p} \circ \overline{s} = \mathrm{id}_M$. Note that a section always exists.

Let $(s, \overline{s})$ be a section of (\ref{abelian-diagram}). We define maps $l_B : A \otimes B \rightarrow B,~ (a, b) \mapsto a \cdot_B b$ and $r_B : B \otimes A \rightarrow B, ~(b, a) \mapsto b \cdot_B a$ by
\begin{align*}
a \cdot_B b = s(a) \cdot_{\widehat{A}} i(b) \quad \text{ and } \quad 
b \cdot_B a = i(b) \cdot_{\widehat{A}} s(a).
\end{align*}
Here $~\cdot_{\widehat{A}}~$ denotes the multiplication on the associative algebra $\widehat{A}$. It is easy to see that the maps $l_B, r_B$ defines an $A$-bimodule structure on $B$. Similarly, we define maps $l_N : A \otimes N \rightarrow N, ~(a, n) \mapsto a \cdot_N n$ and $r_N : N \otimes A \rightarrow N, (n,a) \mapsto n \cdot_N a$ by
\begin{align*}
a \cdot_N n = s(a) \cdot_{\widehat{M}} \overline{i}(n) \quad \text{ and } \quad n \cdot_N a =  \overline{i}(n) \cdot_{\widehat{M}} s(a).
\end{align*}
Here $\cdot_{\widehat{M}}$ denotes both left and right $\widehat{A}$-actions on $\widehat{M}$. The maps $l_N, r_N$ defines an $A$-bimodule structure on $N$. We also define maps $l : M \otimes B \rightarrow N$ and $r: B \otimes M \rightarrow N$ by
\begin{align*}
l (m,b) =   \overline{s}(m) \cdot_{\widehat{M}} i(b)   \quad \text{ and } \quad  r (b,m) =   i(b) \cdot_{\widehat{M}} \overline{s}(m).
\end{align*}
It simply follows that the maps $l$ and $r$ satisfy the conditions of (\ref{bimod-l}) and (\ref{bimod-r}). Finally, we observe that
\begin{align*}
&R(m) \cdot_B S(n) && S(n) \cdot_B R(m)\\
&= sR(m) \cdot_{\widehat{A}} iS(n) && =iS(n) \cdot_{\widehat{A}} sR(m)\\
&= \widehat{R} (\overline{s}(m)) \cdot_{\widehat{A}} \widehat{R}(\overline{i}(n)) && =\widehat{R} (\overline{i}(n)) \cdot_{\widehat{A}} \widehat{R}(\overline{s}(m))\\
&= \widehat{R} \big( \widehat{R}(\overline{s}(m)) \cdot_{\widehat{M}} \overline{i}(n) ~+~ \overline{s}(m) \cdot_{\widehat{M}} \widehat{R} (\overline{i}(n)) \big) && = \widehat{R} \big( \widehat{R} (\overline{i}(n)) \cdot_{\widehat{M}} \overline{s}(m) ~+~ \overline{i}(n) \cdot_{\widehat{M}}  \widehat{R}(\overline{s}(m))    \big)\\
&= \widehat{R} \big( sR(m) \cdot_{\widehat{M}} \overline{i}(n) ~+~ \overline{s}(m) \cdot_{\widehat{M}} iS(n)  \big) && = \widehat{R} \big( {i}S(n) \cdot_{\widehat{M}} \overline{s}(m) ~+~ \overline{i}(n) \cdot_{\widehat{M}} {s}R(m)   \big)\\
&= S \big( R(m) \cdot_N n + l (m, S(n)) \big), && = S \big( r (S(n), m) + n \cdot_N R(m)  \big).
\end{align*}
Thus, $(N \xrightarrow{S} B, l, r)$ defines a bimodule over the relative Rota-Baxter algebra.

Next, we claim that this bimodule is independent of the choice of the section $(s, \overline{s})$ of (\ref{abelian-diagram}). To prove this, take another section $(s', \overline{s}')$ of (\ref{abelian-diagram}).  We first observe that
\begin{align*}
s(a) - s'(a) \in \mathrm{ker ~} p = \mathrm{im ~}i ~~~ \text{ and } ~~~ \overline{s}(m) - \overline{s}'(m) \in  \mathrm{ker ~}\overline{p} = \mathrm{im ~}\overline{i}, ~ \text{ for } a \in A, ~ m \in M.
\end{align*}
Hence we have
\begin{align*}
a \cdot_B b - a \cdot'_B b = (s(a) - s'(a)) \cdot_{\widehat{A}} i(b) = 0 ~~~ \text{ and } ~~~ b \cdot_B a - b \cdot'_B a = i(b) \cdot_{\widehat{A}} (s(a) - s'(a)) = 0.
\end{align*}
Here $\cdot'_B$ denotes the $A$-bimodule structure on $B$ induced by the section $(s', \overline{s}')$. It follows that the $A$-bimodule structure on $B$ does not depend on the choice of section. Similarly,
\begin{align*}
a \cdot_N n - a \cdot'_N n = (s(a) - s'(a)) \cdot_{\widehat{M}} \overline{i}(n) = 0 ~~~~ \text{~ and ~} ~~~~ n \cdot_N a - n \cdot'_N a = \overline{i}(n) \cdot_{\widehat{M}} (s(a) - s'(a)) = 0
\end{align*}
which shows that the $A$-bimodule structure on $N$ is independent of the choice of section. Finally, if $l':M\otimes B\rightarrow N$ and $r':B\otimes M\rightarrow N$ denote the bilinear maps induced by the section $(s', \overline{s}')$ then
\begin{align*}
l (m, b) - l'(m, b) = (\overline{s}(m) - \overline{s}'(m)) \cdot_{\widehat{M}} i(b) = 0 ~~~ \text{ and } ~~~ r (b,m) - r'(b,m) = i(b) \cdot_{\widehat{M}} (\overline{s}(m) - \overline{s}'(m)) = 0.
\end{align*}
Hence the bilinear maps $l:M\otimes B \rightarrow N$ and $r:B\otimes M\rightarrow N$ are also independent of the choice of section, which proves our claim.

\begin{defn}\label{iso-abelian}
Let $\widehat{M} \xrightarrow{\widehat{R}} \widehat{A}$ and $\widehat{M}' \xrightarrow{\widehat{R}'} \widehat{A}'$ be two abelian extensions of the relative Rota-Baxter algebra $M \xrightarrow{R} A$ by the $2$-term chain complex $N \xrightarrow{S} B$. They are said to be {\bf isomorphic} if there is an isomorphism $(\phi, \psi)$ of relative Rota-Baxter algebras from $\widehat{M} \xrightarrow{\widehat{R}} \widehat{A}$ to $\widehat{M}' \xrightarrow{\widehat{R}'} \widehat{A}'$ making the following diagram commutative
\[
\xymatrixrowsep{0.36cm}
\xymatrixcolsep{0.36cm}
\xymatrix{
0 \ar[rr] &  & N \ar[rr] \ar[dd] \ar@{=}[rd] & & \widehat{M} \ar[rr] \ar[rd]^\psi \ar[dd] & & M \ar[dd] \ar[rr] \ar@{=}[rd] & & 0 \\
 & 0 \ar[rr] & & N \ar[rr] \ar[dd] & & \widehat{M}' \ar[rr] \ar[dd] & & M \ar[rr] \ar[dd] & & 0 \\
0 \ar[rr] &  & B \ar[rr] \ar@{=}[rd] & & \widehat{A} \ar[rr] \ar[rd]^\phi & & A \ar[rr] \ar@{=}[rd] & & 0 \\
 & 0 \ar[rr] & & B \ar[rr] & & \widehat{A}' \ar[rr] & & A \ar[rr] & & 0. \\
}
\]
\end{defn}

Let $M \xrightarrow{R} A$ be a relative Rota-Baxter algebra and $(N \xrightarrow{S} B, l, r)$ be a bimodule over it. We denote by $\mathrm{Ext}( M \xrightarrow{R} A ; N \xrightarrow{S} B)$ the set of isomorphism classes of abelian extensions of the relative Rota-Baxter algebra $M \xrightarrow{R} A$ by the $2$-term chain complex $N \xrightarrow{S} B$ so that the induced bimodule structure on $N \xrightarrow{S} B$ is the prescribed one.

With these notations, we have the following.

\begin{thm}\label{abelian-ext-thm}
Let $M \xrightarrow{R} A$ be a relative Rota-Baxter algebra and  $(N \xrightarrow{S} B, l, r)$ be a bimodule over it. Then there is a one-to-one correspondence between $\mathrm{Ext}( M \xrightarrow{R} A ; N \xrightarrow{S} B)$  and the second cohomology group $H^2_\mathsf{rRB}( M \xrightarrow{R} A ; N \xrightarrow{S} B)$.
\end{thm}

\begin{proof}
Let (\ref{abelian-diagram}) be an abelian extension of $M \xrightarrow{R} A$ by the $2$-term chain complex $N \xrightarrow{S} B$. For any section $(s, \overline{s})$, we define maps 
\begin{align*}
&\alpha \in \mathrm{Hom}(A^{\otimes 2}, B), \quad \alpha(a, a') := s(a) \cdot_{\widehat{A}} s(a') - s (a \cdot a'),\\
&\beta \in \mathrm{Hom} (\mathcal{A}^{1,1}, N), \quad \begin{cases}
\beta (a,m) := s(a) \cdot_{\widehat{M}} \overline{s}(m) - \overline{s}(a \cdot_M m),\\
\beta (m,a) := \overline{s}(m) \cdot_{\widehat{M}} s(a) - \overline{s} (m \cdot_M a),\\
\end{cases}\\
&\gamma \in \mathrm{Hom}(M,B), \qquad \quad \gamma (m) := \widehat{R}(\overline{s}(m)) - s (R(m)). 
\end{align*}
By a straightforward calculation, it follows that $(\alpha, \beta, \gamma)$ is a $2$-cocycle in $Z^2_\mathsf{rRB} (M \xrightarrow{R} A; N \xrightarrow{S} B)$. Hence the abelian extension (\ref{abelian-diagram}) corresponds to a cohomology class in $H^2_\mathsf{rRB} (M \xrightarrow{R} A; N \xrightarrow{S} B)$. Moreover, the cohomology class does not depend on the choice of section.

Next, let $\widehat{M} \xrightarrow{\widehat{R}} \widehat{A}$ and $\widehat{M}' \xrightarrow{\widehat{R}'} \widehat{A}'$ be two isomorphic abelian extensions and the isomorphism is given by $(\phi, \psi)$ (see Definition \ref{iso-abelian}). Let $(s, \overline{s})$ be a section of the first abelian extension. Then we have
\begin{align*}
p' \circ (\phi \circ s) = p \circ s = \mathrm{id}_A ~~~~ \text{ and } ~~~~ \overline{p}' \circ (\psi \circ \overline{s}) = \overline{p} \circ \overline{s} = \mathrm{id}_M.
\end{align*}
Therefore, $(\phi \circ s, \psi \circ \overline{s})$ is a section of the second abelian extension.  If $(\alpha', \beta', \gamma')$ denotes the $2$-cocycle in $Z^2_\mathsf{rRB} (M \xrightarrow{R} A; N \xrightarrow{S} B)$ corresponding to the second abelian extension and its section $(\phi \circ s, \psi \circ \overline{s})$, then 
\begin{align*}
\alpha' (a, a') =~& (\phi \circ s)(a) \cdot_{\widehat{A}'} (\phi \circ s)(a') - (\phi \circ s) (a \cdot a') \\
=~& \phi \big(  s(a) \cdot_{\widehat{A}} s(a') - s(a \cdot a') \big)  \\
=~& \phi (\alpha (a,a')) = \alpha (a, a') \quad (\because \phi|_B = \mathrm{id}_B).
\end{align*}
Similarly, one can show that $\beta = \beta'$ and $\gamma = \gamma'$. So, $(\alpha, \beta, \gamma)$ and $(\alpha', \beta', \gamma')$ corresponds to the same element in $H^2_\mathsf{rRB} (M \xrightarrow{R} A; N \xrightarrow{S} B)$. Hence there is a well-defined map
\begin{align*}
\Theta_1 : \mathrm{Ext}(M \xrightarrow{R} A; N \xrightarrow{S} B) ~\rightarrow~ H^2_\mathsf{rRB} (M \xrightarrow{R} A; N \xrightarrow{S} B).
\end{align*}

\medskip

To define the map in the other direction, we take a $2$-cocycle $(\alpha, \beta, \gamma) \in Z^2_\mathsf{rRB} (M \xrightarrow{R} A; N \xrightarrow{S} B)$. In other words, we have
\begin{align}\label{2-co-equiv}
\delta_{A, B}(\alpha) = 0, \quad \delta_{A, N}^\alpha (\beta) = 0 ~~~~ \text{ and } ~~~~ \delta_{M, B} (\gamma) + h_R (\alpha , \beta) = 0.
\end{align}
Let $\widehat{A} = A \oplus B$ and $\widehat{M} = M \oplus N$. We define maps
\begin{align*}
\cdot_{\widehat{A}} : \widehat{A} \otimes \widehat{A} \rightarrow \widehat{A}, \qquad \cdot_{\widehat{M}} : \widehat{A} \otimes \widehat{M} \rightarrow \widehat{ M} \quad \text{ and } \quad \cdot_{\widehat{M}} : \widehat{M} \otimes \widehat{A} \rightarrow \widehat{M} ~~~~ \text{ by }
\end{align*}
\begin{align*}
(a, b) \cdot_{\widehat{A}} (a', b') =~& (a \cdot a',~ a \cdot_B b' + b \cdot_B a' + \alpha (a, a')),\\
(a, b) \cdot_{\widehat{M}} (m,n) =~& (a \cdot_M m, ~a \cdot_N n + r (b,m) + \beta (a,m)),\\
(m,n) \cdot_{\widehat{M}} (a, b) =~& (m \cdot_M a, ~ l (m,b) + n \cdot_N a + \beta (m, a)).
\end{align*}
Since $\alpha$ and $\beta$ satisfies the first two conditions of (\ref{2-co-equiv}), it follows that the above structure maps defines an associative algebra structure on $\widehat{A}$ and an $\widehat{A}$-bimodule structure on $\widehat{M}$. Finally, we define a map $\widehat{R} : \widehat{M} \rightarrow \widehat{A}$ by
\begin{align*}
\widehat{R}(m,n) = (R(m),~ S(n) + \gamma (m)), ~\text{ for } (m,n) \in \widehat{M}.
\end{align*}
It is also straightforward to see that $\widehat{R} : \widehat{M} \rightarrow \widehat{A}$ is a relative Rota-Baxter operator (follows from the last condition of (\ref{2-co-equiv})). In other words, $\widehat{M} \xrightarrow{\widehat{R}} \widehat{A}$ is a relative Rota-Baxter algebra.  Moreover, it is an abelian extension of $M \xrightarrow{R} A$ by the $2$-term chain complex $N \xrightarrow{S} B$ as of (\ref{abelian-diagram}) with the structure maps
\begin{align*}
i (b) = (0, b), \quad \overline{i}(n) = (0,n), \quad p(a,b) = a ~~~ \text{ and } ~~~ \overline{p}(m,n) = m.
\end{align*}
Let $(\alpha', \beta', \gamma') \in Z^2_\mathsf{rRB} (M \xrightarrow{R} A; N \xrightarrow{S} B)$ be another $2$-cocycle cohomologous to $(\alpha, \beta, \gamma)$, say
\begin{align*}
(\alpha, \beta, \gamma) - (\alpha', \beta', \gamma') = \delta_\mathsf{rRB} (\theta, \vartheta),
\end{align*}
for some $(\theta, \vartheta) \in C^1_\mathsf{rRB} (M \xrightarrow{R} A; N \xrightarrow{S} B).$ We define maps $\phi : A \oplus B \rightarrow A \oplus B$ and $\psi : M \oplus N \rightarrow M \oplus N$ by
\begin{align*}
\phi (a,b) =~& (a,~ b + \theta (a)),\\
\psi (m,n) =~& (m, ~ n + \vartheta (m)).
\end{align*}
Then it is easy to check that $(\phi, \psi)$ defines an isomorphism of abelian extensions from $\widehat{M}  \xrightarrow{\widehat{R}} \widehat{A}$ to $\widehat{M}'  \xrightarrow{\widehat{R}'} \widehat{A}'$. Thus, there is a well-defined map
\begin{align*}
\Theta_2 : H^2_\mathsf{rRB} (M \xrightarrow{R} A; N \xrightarrow{S} B) ~\rightarrow ~ \mathrm{Ext} (M \xrightarrow{R} A; N \xrightarrow{S} B).
\end{align*}
Finally, the maps $\Theta_1$ and $\Theta_2$ are inverses to each other. This completes the proof.
\end{proof}

\section{Classifications of skeletal homotopy relative Rota-Baxter algebras}\label{sec-6}

The notion of homotopy relative Rota-Baxter Lie algebras was introduced by Lazarev, Sheng, and Tang \cite{laza-rota}. Homotopy theory of relative Rota-Baxter associative algebras is studied in \cite{DasSK}. Roughly, a homotopy relative Rota-Baxter (associative) algebra is a triple consisting of an $A_\infty$-algebra, a bimodule over the $A_\infty$-algebra and a homotopy relative Rota-Baxter operator. In this section, we mainly focus on homotopy relative Rota-Baxter algebras whose underlying $A_\infty$-algebra and the bimodule are both concentrated in degrees $0$ and $1$. We call them $2$-term homotopy relative Rota-Baxter algebras. We classify `skeletal' homotopy relative Rota-Baxter algebras.

\begin{defn}\label{2t}\cite{stas}
A {\bf $2$-term $A_\infty$-algebra} $\mathcal{A} =(A_1 \xrightarrow{d} A_0, \mu_2, \mu_3)$ consists of a chain complex $A_1 \xrightarrow{d} A_0$ together with bilinear maps $\mu_2 : A_i \otimes A_j \rightarrow A_{i+j}$ $(0 \leq i, j, i+j \leq 1)$ and a trilinear map $\mu_3 : A_0 \otimes A_0 \otimes A_0 \rightarrow A_1$ satisfying the following set of identities: for $a, b, c, e \in A_0$ and $p, q\in A_1$,
\begin{itemize}
\item[(i)] $d \mu_2 (a, p) = \mu_2 (a, dp),$\\
$d \mu_2 (p, a) = \mu_2 (dp, a),$ 
\item[(ii)] $\mu_2 (dp, q) = \mu_2 (p, dq),$
\item[(iii)] $d \mu_3 (a, b, c) = \mu_2 (\mu_2 (a, b), c) - \mu_2 (a, \mu_2 (b, c)),$\\
$\mu_3 (a, b, dp) = \mu_2 (\mu_2 (a, b), p) - \mu_2 (a, \mu_2 (b, p)),$\\
$\mu_3 (a, dp, c) = \mu_2 (\mu_2 (a, p), c) - \mu_2 (a, \mu_2 (p, c))$,\\
$\mu_3 (dp, b, c) = \mu_2 (\mu_2 (p, b), c) - \mu_2 (p, \mu_2 (b, c)),$
\item[(iv)] $\mu_3 (\mu_2 (a, b), c, e) - \mu_3 (a, \mu_2 (b, c), e) + \mu_3 (a, b, \mu_2 (c, e)) = \mu_2 (\mu_3(a, b, c), e) + \mu_2 (a, \mu_3 (b, c, e)).$
\end{itemize}
\end{defn}

The notion of bimodules over a $2$-term $A_\infty$-algebra is a particular case of bimodules over an $A_\infty$-algebra \cite{keller}.

\begin{defn}
Let $\mathcal{A} = (A_1 \xrightarrow{d} A_0, \mu_2, \mu_2)$ be a $2$-term $A_\infty$-algebra. A {\bf bimodule} over it consists of a tuple $\mathcal{M} = (M_1 \xrightarrow{d^\mathcal{M}} M_0, \mu_2^\mathcal{M}, \mu_3^\mathcal{M})$ of a chain complex $M_1 \xrightarrow{d^\mathcal{M}} M_0$ together with bilinear maps 
\begin{align*}
\mu_2^\mathcal{M} : A_i \otimes M_j \rightarrow M_{i+j}, \qquad
\mu_2^\mathcal{M} : M_i \otimes A_j \rightarrow M_{i+j} \quad (0 \leq i , j, i+j \leq 1)
\end{align*}
and a trilinear map $\mu_3^\mathcal{M} : \mathcal{A}_0^{2,1} \rightarrow M_1$ \big(here $\mathcal{A}_0^{2,1} = (A_0 \otimes A_0 \otimes M_0) \oplus (A_0 \otimes M_0 \otimes A_0) \oplus (M_0 \otimes A_0 \otimes A_0)$\big)
satisfying the all possible identities corresponding to (i)-(iv) of Definition \ref{2t} where exactly one of the variables in each identity is replaced by an element of $\mathcal{M}$ and the corresponding operation $d, \mu_2$ or $\mu_3$ is replaced by $d^\mathcal{M}, \mu_2^\mathcal{M}$ or $\mu_3^\mathcal{M}$.
\end{defn}

\begin{exam}
Any $2$-term $A_\infty$-algebra $\mathcal{A} = (A_1 \xrightarrow{d} A_0, \mu_2, \mu_3)$ is a bimodule over itself. This is called the adjoint bimodule.
\end{exam}

\begin{defn}\label{defn-hrrbo}
Let $\mathcal{A} = (A_1 \xrightarrow{d} A_0, \mu_2, \mu_3)$ be a $2$-term $A_\infty$-algebra and $\mathcal{M} = (M_1 \xrightarrow{d^\mathcal{M}} M_0, \mu_2^\mathcal{M}, \mu_3^\mathcal{M})$ be a bimodule over it. A ($2$-term) {\bf homotopy relative Rota-Baxter operator} (on $\mathcal{A}$ with respect to $\mathcal{M}$) is a triple $\mathcal{R} = (R_0, R_1, R_2)$ of maps
\begin{align*}
R_0 : M_0 \rightarrow A_0, \qquad R_1 : M_1 \rightarrow A_1 ~~~ \text{ and } ~~~ R_2 : M_0 \otimes M_0 \rightarrow A_1
\end{align*} 
satisfying
\begin{itemize}
\item[(i)] $d \circ R_1 = R_0 \circ d^\mathcal{M}$,
\item[(ii)] $R_0 \big( \mu_2^\mathcal{M} (R_0 (m), m') + \mu_2^\mathcal{M} (m, R_0 (m')) \big) - \mu_2 (R_0 (m), R_0 (m')) = d (R_2 (m, m'))$, 
\item[(iii)] $R_1 \big(  \mu_2^\mathcal{M} (R_0 (m), n) + \mu_2^\mathcal{M} (m, R_1 (n)) \big) - \mu_2 (R_0 (m), R_1 (n)) = R_2 (m, d^\mathcal{M} (n))$,
\item[(iv)] $ R_1 \big(  \mu_2^\mathcal{M} (R_1 (n), m) + \mu_2^\mathcal{M} (n, R_0 (m))  \big) - \mu_2 (R_1 (n), R_0 (m)) = R_2 (d^\mathcal{M} (n), m)$,
\item[(v)] $ \mu_2 (R_0 (m), R_2 (m', m'')) - R_1 (\mu_2^\mathcal{M} (m, R_2 (m', m'')) ) - R_2 \big( \mu_2^\mathcal{M} (R_0 (m), m') + \mu_2^\mathcal{M} (m, R_0 (m')),~ m''  \big) \\
+ R_2 \big(m,~ \mu_2^\mathcal{M} ( R_0 (m') , m'') + \mu_2^\mathcal{M} (m', R_0 (m'') )   \big) - \mu_2 (R_2 (m, m'), R_0 (m'')) + R_1 (\mu_2^\mathcal{M} (R_2 (m, m'), m'')) \\
+ \mu_3^\mathcal{M} (m, R_0 (m'), R_0 (m'')) + \mu_3^\mathcal{M} (R_0 (m), m', R_0 (m'')) + \mu_3^\mathcal{M} (R_0 (m), R_0 (m'), m'') \\
= \mu_3 (R_0 (m), R_0 (m'), R_0 (m'')),$
\end{itemize}
for $m, m', m'' \in M_0$ and $n \in M_1$.
\end{defn}

A  ($2$-term) {\bf homotopy relative Rota-Baxter algebra} is a triple $(\mathcal{A}, \mathcal{M}, \mathcal{R})$ in which $\mathcal{A}$ is a $2$-term $A_\infty$-algebra, $\mathcal{M}$ is an $\mathcal{A}$-bimodule and $\mathcal{R}$ is a homotopy relative Rota-Baxter operator. Like nonhomotopic case, we denote a homotopy relative Rota-Baxter algebra as above by $\mathcal{M} \xrightarrow{\mathcal{R}} \mathcal{A}$.

\begin{remark}
Let $\mathcal{R}$ be a ($2$-term) homotopy relative Rota-Baxter operator on an associative $2$-algebra $\mathcal{A}$ with respect to a bimodule $\mathcal{M}$. Then the standard skew-symmetrization process yields a relative Rota-Baxter operator (in the sense of \cite{Sheng}) on the associated Lie $2$-algebra with respect to the representation. For more details, we refer to \cite[Proposition 7.8]{DasSK}.
\end{remark}

Next, we consider skeletal homotopy relative Rota-Baxter algebras and classify them in terms of $3$-cocycles in the cochain complex of relative Rota-Baxter algebras.

\begin{defn}
A homotopy relative Rota-Baxter algebra $\mathcal{M} \xrightarrow{\mathcal{R}} \mathcal{A}$ with
\begin{align*}
\mathcal{A} = (A_1 \xrightarrow{d} A_0, \mu_2, \mu_3), \quad \mathcal{M} = (M_1 \xrightarrow{d^\mathcal{M}} M_0, \mu_2^\mathcal{M}, \mu_3^\mathcal{M}) ~~~ \text{ and } ~~~ \mathcal{R} = (R_0, R_1, R_2)
\end{align*}
is said to be {\bf skeletal} if $d = 0$ and $d^\mathcal{M} = 0$.
\end{defn}

\begin{thm}\label{skeletal-thm}
There is a one-to-one correspondence between skeletal homotopy relative Rota-Baxter algebras and triples $(M \xrightarrow{R} A; N \xrightarrow{S} B ; (\alpha, \beta, \gamma))$ in which $M \xrightarrow{R} A$ is a relative Rota-Baxter algebra, $N \xrightarrow{S} B$ is a bimodule and $(\alpha, \beta, \gamma) \in Z^3_\mathsf{rRB} (M \xrightarrow{R} A ; N \xrightarrow{S} B)$ is a $3$-cocycle.
\end{thm}

\begin{proof}
Let $\mathcal{M} \xrightarrow{\mathcal{R}} \mathcal{A}$ be a skeletal homotopy relative Rota-Baxter algebra, where
\begin{align*}
\mathcal{A} = (A_1 \xrightarrow{0} A_0, \mu_2, \mu_3), \quad \mathcal{M} = (M_1 \xrightarrow{0} M_0, \mu_2^\mathcal{M}, \mu_3^\mathcal{M}) ~~~ \text{ and } ~~~ \mathcal{R} = (R_0, R_1, R_2).
\end{align*}
Since $\mathcal{A} = (A_1 \xrightarrow{0} A_0, \mu_2, \mu_3)$ is a $2$-term $A_\infty$-algebra, it follows from Definition \ref{2t} that $A_0$ is an associative algebra with the multiplication $\mu (a, b) := \mu_2 (a, b)$, for $a, b \in A_0$, and $A_1$ is a bimodule over the associative algebra $A_0$ with left and right $A_0$-actions $l_{A_1} : A_0 \otimes A_1 \rightarrow A_1$ and $r_{A_1} : A_1 \otimes A_0 \rightarrow A_1$ given by
\begin{align*}
l_{A_1} (a, p) := \mu_2 (a, p), \quad r_{A_1} (p, a) := \mu_2 (p, a), ~ \text{ for } a \in A_0, p \in A_1.
\end{align*}
Further, the map $\mu_3 \in \mathrm{Hom} (A_0^{\otimes 3}, A_1)$ is a $3$-cocycle in the Hochschild cochain complex of $A_0$ with coefficients in the $A_0$-bimodule $A_1$. Moreover, $\mathcal{M} = (M_1 \xrightarrow{0} M_0, \mu_2^\mathcal{M}, \mu_3^\mathcal{M})$ is a bimodule over $\mathcal{A}$ implies that both $M_0$ and $M_1$ are bimodules over the associative algebra $A_0$ with left and right actions given by
\begin{align*}
&l_{M_0} (a, m) = \mu_2^\mathcal{M} (a, m), \quad r_{M_0} (m, a) = \mu_2^\mathcal{M} (m, a), ~ \text{ for } a \in A_0, m \in M_0, \\
&l_{M_1} (a, n) = \mu_2^\mathcal{M} (a, n), \quad r_{M_1} (n, a) = \mu_2^\mathcal{M} (n, a), ~ \text{ for } a\in A_0, n \in M_1.
\end{align*}
Further, we define maps $l : M_0 \otimes A_1 \rightarrow M_1$ and $r : A_1 \otimes M_0 \rightarrow M_1$ by $l (m, p) = \mu_2^\mathcal{M} (m, p)$ and $r (p, m) = \mu_2^\mathcal{M} (p, m)$, for $m \in M_0$, $p \in A_1$. It is also easy to verify that the map $\mu_3^\mathcal{M} \in \mathrm{Hom}(\mathcal{A}_0^{2,1}, M_1)$ satisfies 
$\delta^{\mu_3}_{A_0, M_1} (\mu_3^\mathcal{M}) = 0$.

Finally, since $\mathcal{R} = (R_0, R_1, R_2)$ is a homotopy relative Rota-Baxter operator, it follows from the conditions (ii)-(iv) of Definition \ref{defn-hrrbo} that $M_0 \xrightarrow{R_0} A_0$ is a relative Rota-Baxter algebra and $(M_1 \xrightarrow{R_1} A_1, l, r)$ is a bimodule over it. Finally, the condition (v) of Definition \ref{defn-hrrbo} implies that $\delta_{M_0, A_1} (R_2) + h_{R_0} (\mu_3, \mu_3^\mathcal{M}) = 0$. Hence, we have
\begin{align*}
{\delta_\mathsf{rRB}} ((\mu_3, \mu_3^\mathcal{M}, R_2 )) = \big( \delta_{A_0, A_1} (\mu_3), ~ \delta^{\mu_3}_{A_0, M_1} (\mu_3^\mathcal{M}),~ \delta_{M_0, A_1} (R_2) + h_{R_0} (\mu_3, \mu_3^\mathcal{M})  \big) = 0.
\end{align*}
This shows that $(M_0 \xrightarrow{R_0} A_0 ; M_1 \xrightarrow{R_1} A_1 ; (\mu_3, \mu_3^\mathcal{M}, R_2))$ is a required triple.

\medskip

Conversely, let $(M \xrightarrow{R} A; N \xrightarrow{S} B ; (\alpha, \beta, \gamma))$ be a triple consisting of a relative Rota-Baxter algebra $M \xrightarrow{R} A$, a bimodule $(N \xrightarrow{S} B, l, r)$ and a $3$-cocycle
\begin{align*}
(\alpha, \beta, \gamma) \in Z^3_\mathsf{rRB} (M \xrightarrow{R} A; N \xrightarrow{S} B) \subset \mathrm{Hom} (A^{\otimes 3}, B) \oplus \mathrm{Hom}(\mathcal{A}^{2,1}, N) \oplus \mathrm{Hom} (M^{\otimes 2}, B).
\end{align*}
Then it is easy to verify that $\mathcal{A} = (B \xrightarrow{0} A, \mu_2, \mu_3)$ is a $2$-term $A_\infty$-algebra, where $\mu_2$ and $\mu_3$ are given by
\begin{center}
$\mu_2 (a, a') = a \cdot a', \quad \mu_2 (a, b) = a \cdot_B b, \quad \mu_2 (b, a) = b \cdot_B a, ~\text{ for } a, a' \in A, b \in B,$ \\
$\mu_3 (a, a', a'') = \alpha (a, a', a''), \text{ for } a, a', a'' \in A.$
\end{center}
Moreover, $\mathcal{M} = (N \xrightarrow{0} M, \mu_2^\mathcal{M}, \mu_3^\mathcal{M})$ is a bimodule over $\mathcal{A}$, where
\begin{align*}
&\mu_2^\mathcal{M} (a, m)= a \cdot_M m, \quad \mu_2^\mathcal{M} (a, n) = a \cdot_N n, \quad \mu_2^\mathcal{M} (b, m) = r (b, m),\\
&\mu_2^\mathcal{M} (m, a) = m \cdot_M a, \quad \mu_2^\mathcal{M} (n, a) = n \cdot_N a, \quad \mu_2^\mathcal{M} (m, b) = l(m, b),
\end{align*}
for $a \in A,~ b \in B,~ m \in M,~ n \in N$, and $\mu_3^\mathcal{M} : \mathcal{A}^{2,1} \rightarrow N$ is given by $\mu_3^\mathcal{M} = \beta$. Further, one can show that $\mathcal{R} = (R, S, \gamma)$ is a homotopy relative Rota-baxter operator. In other words, $\mathcal{M} \xrightarrow{\mathcal{R}} \mathcal{A}$ is a (skeletal) homotopy relative Rota-Baxter algebra.

The above two correspondences are inverses to each other, which completes the proof.
\end{proof}

\medskip

\noindent {\bf Acknowledgements.} We thank the anonymous referee for his useful comments and suggestions, which improved the presentation of the article. A. Das would like to thank IIT Kharagpur (India) for providing a beautiful academic atmosphere where some parts of the research were carried out. Some part of the research work of S. K. Mishra was supported by the NBHM postdoctoral fellowship, India.

\noindent {\bf Data availability.} Data sharing not applicable to this article as no datasets were generated or analysed during the current study.

\noindent {\bf Funding.} The research work of S. K. Mishra was supported by NBHM Postdoctoral Fellowship (Grant number 0204/8/2020/R\&D-II/9976).

\noindent {\bf Compliance with Ethical Standards.}
The authors have no relevant financial or non-financial interests to disclose. 


\begin{thebibliography}{BFGM03}



\bibitem{aguiar-pre} Aguiar, M.: Pre-Poisson algebras. Lett. Math. Phys. 54, 263--277 (2000).





\bibitem{atkinson} Atkinson, F.V.: Some aspects of Baxter's functional equation. {J. Math. Anal. Appl.} 7, 1--30 (1963).



\bibitem{bai-splitting} Bai, C., Bellier, O., Guo, L.: Splitting of operations, Manin products, and Rota-Baxter operators. 	Int. Math. Res. Not. IMRN 2013 (3), 485--524 (2013).

\bibitem{bai-o} Bai, C., Guo, L., Ni, X.: $\mathcal{O}$-operators on associative algebras and dendriform algebras. { J. Algebra Appl.} 12, 1350027 (2013).



\bibitem{baxter} Baxter, G.: An analytic problem whose solution follows from a simple algebraic identity. { Pacific J. Math.} 10, 731--742 (1960).



\bibitem{cartier} Cartier, P.: On the structure of free Baxter algebras. Adv. Math. 9, 253--265 (1972).





\bibitem{conn}
Connes, A., Kreimer, D.: Renormalization in quantum field theory and the Riemann-Hilbert problem. I. The Hopf algebra structure of graphs and the main theorem.
{ Comm. Math. Phys.} 210 (1), 249--273  (2000).

\bibitem{das-dend} Das, A.: Cohomology and deformations of dendriform algebras, and ${Dend}_\infty$-algebras. {Comm. Algebra} 50 (4), 1544-1567 (2022).

\bibitem{das-rota} Das, A.: Deformations of associative Rota-Baxter operators. { J. Algebra} 560, 144--180 (2020).

\bibitem{DasSK} 
Das, A., Mishra, S.K.: The $L_{\infty}$-deformations of associative Rota-Baxter algebras and homotopy Rota-Baxter
operators. J. Math. Phys. 63, 051703 (2022).



















\bibitem{gon-kol} Goncharov, M.E., Kolesnikov, P.S.: Simple finite-dimensional double algebras. {J. Algebra} 500, 425--438 (2018).

\bibitem{Guan}
Guan, A., Lazarev, A., Sheng, Y., Tang, R.: Review of deformation theory I: Concrete formulas for deformations of algebraic structures. Adv. Math. (China) 49 (3), 257--277
(2020).


\bibitem{Guan2}
Guan, A., Lazarev, A., Sheng, Y., Tang, R.: Review of deformation theory II: a homotopical approach. Adv. Math. (China) 49 (3), 278--298 (2020).

\bibitem{guo-book} Guo, L.: An introduction to Rota-Baxter algebra. Surveys of Modern Mathematics, 4. { International Press, Somerville, MA; Higher Education Press, Beijing} (2012).

\bibitem{guo-adv} Guo, L., Keigher, W.: Baxter algebras and Shuffle products. {Adv. Math.} 150, 117-149 (2000).

\bibitem{GuoLin}
Guo, L., Lin, Z.: Representations and modules of Rota-Baxter algebras. arXiv:1905.01531


\bibitem{jiang-sheng} Jiang, J., Sheng, Y.: Representations and cohomologies of relative Rota-Baxter Lie algebras and applications. J. Algebra 602, 637--670 (2022).


\bibitem{keller} Keller, B.: Introduction to $A$-infinity algebras and modules. {Homology Homotopy Appl.} 3 (1), 1--35 (2001).


\bibitem{kuper} Kupershmidt, B.A.: What a classical $r$-matrix really is. {J. Nonlinear Math. Phys.} 6, 448--488 (1999).



\bibitem{laza-rota}
Lazarev, A., Sheng, Y., Tang, R.: Deformations and homotopy theory of relative Rota-Baxter Lie algebras. {Comm. Math. Phys.} 385, 595--631 (2021).

\bibitem{loday} Loday, J.L.: Dialgebras and related operads. pp. 7--66, Lecture Notes in Math., 1763, {em Springer, Berlin} (2001).

\bibitem{loday-cyclic} Loday, J.L.: Cyclic Homology. Springer-Verlag, Grundlehren der mathematischen Wissenschaften, 301 (1992).










\bibitem{rota} Rota, G.C.: Baxter algebras and combinatorial identities I, II. {Bull. Amer. Math. Soc.} 75, 325--329 (1969); ibid 75, 330--334 (1969).


\bibitem{Sheng}
Sheng, Y.: Categorification of pre-Lie algebras and solutions of 2-graded classical Yang-Baxter equations. Theory Appl. Categ. 34, 269--294 (2019).



\bibitem{stas} Stasheff, J.: Homotopy associativity of $H$-spaces II. {Trans. Amer. Math. Soc.} 108, 293--312 (1963).


\bibitem{tang} Tang, R.,  Bai, C., Guo, L., Sheng, Y.: Deformations and their controlling cohomologies of $\mathcal{O}$-operators. {Comm. Math. Phys.} 368 (2),  665--700 (2019).

\bibitem{uchino} Uchino, K.: Quantum analogy of Poisson geometry, related dendriform algebras and Rota-Baxter operators. {Lett. Math. Phys.} 85 (2-3),  91--109 (2008).




\bibitem{wang-zhou} Wang, K., Zhou, G.: Deformations and homotopy theory of Rota-Baxter algebras of any weight. {arXiv:2108.06744}.



\bibitem{zgg} Zhang, T., Gao, X., Guo, L.: Hopf algebras of rooted forests, cocycles, and free Rota-Baxter algebras. {J. Math. Phys.} 57, 101701 (2016).

\end{thebibliography}
\end{document}